\documentclass{amsart}
\usepackage{subfiles}
\usepackage{amsmath,amsthm,amssymb, latexsym, stackrel,geometry,enumerate,bbm, amsfonts, wasysym}
\usepackage[all]{xy}
\usepackage{srcltx}
\usepackage[left,modulo, pagewise]{lineno}
\usepackage{newclude}
\usepackage{tikzsymbols}
\usepackage{pgf,tikz}
\usepackage{mathtools}
\usepackage{xcolor}
\usetikzlibrary{patterns}
\usepackage[color=blue!20!white]{todonotes} %%% Notes aside 
\usepackage{hyperref}
\usepackage{graphicx}
\usepackage{tikz-cd}
\usetikzlibrary{shapes}
\usetikzlibrary{arrows}
\usetikzlibrary{topaths}
\usetikzlibrary{decorations.pathmorphing}
\usetikzlibrary{shadows.blur}

\usetikzlibrary{intersections}
\usetikzlibrary{calc}
\usetikzlibrary{positioning}     
\usepackage{tikz}

\usetikzlibrary{positioning}
\usetikzlibrary{matrix}
\usepackage{float}
\usepackage{subfig}

\theoremstyle{plain}
\newtheorem{thm}{Theorem}[section]
\newtheorem{lem}[thm]{Lemma}
\newtheorem{prop}[thm]{Proposition}

\newtheorem{thm-defn}[thm]{Theorem-Definition}

\theoremstyle{definition}
\newtheorem{defn}[thm]{Definition}

\newtheorem{exmp}[thm]{Example}

\theoremstyle{remark}
\newtheorem{rem}[thm]{Remark}

\newcommand{\SC}{\mathcal D_{sg}}

\newcommand{\dissection}{(S,M,\mathcal O, G)}

\newcommand{\cB}{\mathcal{B}}
\newcommand{\cG}{\mathcal{G}}
\newcommand{\cH}{\mathcal{H}}
\newcommand{\gra}{G}
\newcommand{\orbi}{\mathcal{O}}

\usepackage{color}
\usepackage{ulem}

\newcommand{\yad}[1]{{\textcolor{violet}{#1}}}

\begin{document}

\title[Skew-gentle algebras and orbifolds]{Derived categories of skew-gentle algebras and orbifolds}

\author{Daniel Labardini-Fragoso}
\address{Daniel Labardini-Fragoso\newline
Instituto de Matem\'aticas, Universidad Nacional Aut{\'o}noma de M{\'e}xico, 
Ciudad Universitaria.
04510 Mexico City, M{\'e}xico}
\email{labardini@im.unam.mx}

\author{Sibylle Schroll}
\address{Sibylle Schroll\newline
Department of Mathematics, University of Cologne, 
Weyertal 86-90, 
50931 Cologne,
Germany}
\email{schroll@math.uni-koeln.de}

\author{Yadira Valdivieso}
\address{Yadira Valdivieso\newline
School of Mathematics and Actuarial Science, University of Leicester, 
University Rd, 
Leicester LE1 7RH,
UK}
\email{yvd1@leicester.ac.uk}

\normalem
\thanks{The first author was supported by a \emph{C\'{a}tedra Marcos Moshinsky} and the grants CONACyT-238754 and PAPIIT-IN112519. The second and the third author are supported by the Royal Society through the Newton International Fellowship NIF\textbackslash R1\textbackslash 180959. The second author is supported by the EPSRC  Early Career Fellowship EP/P016294/1. The third author is supported by the European Union's Horizon 2020 research and innovation programme under the Marie Sklodowska-Curie grant agreement No 838316}

\begin{abstract}
Skew-gentle algebras are a generalisation of the well-known class of gentle algebras with which they share many common properties.
In this work, using non-commutative Gr\"obner basis theory, we show that these algebras are strong Koszul and that the Koszul dual is again skew-gentle. 
We give a geometric model of their bounded derived categories in terms of polygonal dissections of surfaces with orbifold points, establishing a correspondence between curves in the orbifold and indecomposable objects. Moreover, we show that the orbifold dissections encode homological properties of skew-gentle algebras such as their singularity categories, their Gorenstein dimensions and derived invariants such as the determinant of their $q$-Cartan matrices. 
\end{abstract}

\maketitle

%\tableofcontents

\section{Introduction}

Derived categories  play an important role in many branches of mathematics such as algebraic geometry and representation theory, where they  provide the proper setting for tilting theory \cite{BB80, Bo81,HR82}.

In general, giving a concrete description of the (bounded) derived category of a finite dimensional algebra is not easy to achieve. However, when the derived category is tame, this is often possible, and  a geometric realisation or a combinatorial description of their indecomposables objects and morphisms has been given for several families of well-known algebras, such as  hereditary algebras of Dynkin type or gentle algebras \cite{BM03, Hap87, HKK17,  LP19, OPS18}.

The derived categories of gentle algebras have been gaining relevance in several branches of mathematics; for example, recently these categories have been linked to homological mirror symmetry, a homological framework developed by Kontsevich \cite{Kon94} to explain the similarities between the symplectic geometry of the so-called $A$-model, and the algebraic geometry of the so-called $B$-model  of certain Calabi-Yau manifolds.
Derived categories of gentle algebras have provided a good understanding of the $A$-model in the mirror symmetry program in the case of surfaces. In particular, a connection between graded gentle algebras and Fukaya categories was established in \cite{HKK17, Boc16}, where collections of formal generators in (partially wrapped) Fukaya categories were constructed whose endomorphism algebras are graded gentle algebras. Conversely, in \cite{LP19, OPS18},  given a homologically smooth graded gentle algebra $A$, a graded surface with stops $(S_A, M_A, \eta_A)$ is constructed, where $S_A$ is an oriented smooth surface with non-empty boundary, $M_A$ is a set of stops on the boundary of $A$ and $\eta_A$ is a line field on $A$, such that the partially wrapped Fukaya category $\mathcal W(S_A, M_A, \eta_A)$ and derived category $\mathcal D(A)$ are equivalent.

 Skew-gentle algebras, skew group algebras of gentle algebras, have recently been related to other areas of mathematics where they have been a tool to prove some interesting results. For instance, in \cite{GLFS16}, triangulations giving rise to skew-gentle Jacobian algebras are used to establish the tameness of the Jacobian algebras associated  to triangulations of surfaces, see \cite{LF09,LF16}. These triangulations in \cite{GLFS16} were used in \cite{QZ17} to study the cluster category of punctured surfaces with non-empty boundary where it is shown that  there is  a bijection between so-called strings in a skew-gentle algebra and tagged curves in the corresponding surface.  In \cite{AP17}  a geometric construction of these  cluster categories via a $\mathbb{Z}/2$ actions on surfaces is given. In \cite{Val16}, skew-gentle algebras were used to prove that Jacobian algebras coming from triangulations of closed surfaces are algebras of exponential growth, an unexpected result uncovering a new class of symmetric tame $\Omega$-periodic algebras. Furthermore, in \cite{ping-zhou-zhu20},
 a geometric model for the module category of a skew-gentle algebra has been given in terms of tagged dissections of surfaces. 

The bounded derived categories of skew-gentle algebras have been studied by several authors beginning with  \cite{BMM03}. In that paper, the authors give a classification of the indecomposable objects in terms of so-called generalised homotopy strings and bands. In \cite{BD04},  another classification of the indecomposable objects is given by using different matrix reduction techniques. 
In \cite{AB19}, a geometric classification of the derived equivalence classes of  skew-gentle algebras is given based on the $\mathbb{Z}/2$-action using  the geometric model  and  results for gentle algebras  in \cite{OPS18, APS19}.

In the present paper, we realise the indecomposable objects of the bounded derived category of a skew-gentle algebra as curves on a surface with orbifold points of order two. Furthermore, using non-commutative Gr\"obner bases theory, we give  a direct proof that  skew-gentle algebras are strong Koszul, a property that is not known to be preserved under skew group action. We show that the Koszul dual is skew-gentle and that its geometric model has the same underlying surface. 

More specifically, we show that there is a bijection between skew-gentle algebras and generalised dissections of surfaces with orbifold points of order two, or simply \emph{orbifold dissection}. This bijection is a natural generalisation from gentle to skew-gentle algebras and it  has also recently been shown in \cite{AB19,  ping-zhou-zhu20}. However, in contrast to \cite{AB19}, where derived equivalences between two skew-gentle algebras are studied in terms of diffeomorphisms of the orbifold working mostly in  its double cover, we work directly in the orbifold  and prove that graded curves in an orbifold dissection coming from a skew-gentle algebra are in bijection with homotopy strings and bands, which by \cite{BMM03} describe the indecomposable objects in the bounded derived category of the algebra. Furthermore, the data of the orbifold dissection contains, on the one hand, the data of a line field in the same way as for gentle algebras in  \cite{AB19, APS19, LP19, OPS18}, and on the other, the homological grading in the derived category. In this paper we will focus on the latter.

Our results suggest that the bounded derived category of a skew-gentle algebra should in fact be a partially wrapped Fukaya category. However, in order to establish this, a complete description of the morphisms in the bounded derived category of a skew-gentle algebra is needed. Unlike for gentle algebras, in the skew-gentle case this is an open problem which we are hoping to address in a forthcoming paper \cite{SV}. 

 We now state our first result. According to the classification in \cite{BMM03}, indecomposable objects in the bounded derived category of a skew-gentle algebra fall within two classes, the so-called string objects and  band objects, the latter coming in infinite families. Using this fact, we show the following.

%The orbifold dissection corresponding to a skew-gentle algebra encodes important information of the algebra itself. For instance, we prove that skew-gentle algebras are Koszul, and that the dual dissection corresponds to the orbifold dissection of the Koszul dual. More precisely, we show the following.

\bigskip
\noindent\textbf{Theorem A} (Theorem~\ref{thm:indecomposable objects}) {\it \, Let $A$ be a skew-gentle algebra with associated surface  $O$ with orbifold points and induced orbifold dissection $\gra$. Denote by $M$ the set of vertices of $\gra$. Then the data of $(O, M, \gra)$ gives a geometric model for the objects of the bounded derived category $D^b(A)$ of $A$. More precisely, 
\begin{enumerate}
    \item the indecomposable string objects in $D^b(A)$ are induced by graded arcs $(\gamma, f)$, where $\gamma$ is an orbifold homotopy class of curves in $O$ that either start and end in marked points in $M$, or start in a marked point in $M$ and wrap around a puncture at the other end, or wrap around a puncture at each end, and where $f$ is a grading on $\gamma$;
    \item each family of indecomposable band objects in $D^b(A)$ corresponds to
    a graded closed curve $(\gamma, f)$ where $\gamma$ is  an orbifold homotopy class of closed curves in $O$ with grading $f$  such that the  combinatorial winding number induced by $f$ is zero.
\end{enumerate}
}

We note that the combinatorial grading on curves naturally encodes the data of a line field on the orbifold surface. 

 Furthermore, we show  that the orbifold dissection corresponding to a skew-gentle algebra encodes important information of the algebra itself. Namely, as an application of our geometric model, we show how the orbifold dissection associated to a skew-gentle algebra $A$ encodes the singularity category of the $A$, its Gorenstein dimension and the  derived invariant given by the $q$-Cartan matrix of $A$. 

In \cite{Green2017} Green introduced the notion of a \emph{strong Koszul algebra}. By definition an algebra is strong Koszul if it has a quadratic Gr\"obner basis. It is shown in \cite{GH95} that strong Koszul algebras are Koszul. However, there are examples of Koszul algebras which are not strong Koszul. One such example is the family of Sklyanin algebras \cite{Smith94}. 

We prove  that skew-gentle algebras are strong Koszul and that their Koszul dual is again a (possibly infinite dimensional) skew-gentle algebra which can be realised on the same surface.  More precisely, we show the following.

\bigskip
\noindent\textbf{Theorem B} (Theorem~\ref{Thm:dual graph}) {\it
Let $A$ be a skew-gentle algebra.
Then $A$ is a strong Koszul algebra and  its Koszul dual $A^!$ is (locally) skew-gentle, and $A$ and $A^!$ give  rise to dual orbifold dissections on the same surface with orbifold points.}

{\bf Acknowledgements:}
The authors would like to thank Claire Amiot and Thomas Br\"ustle for discussing  and sharing the results of their recent paper \cite{AB19} while both their and our papers were still in preparation. The authors would also like to particularly thank Viktor Bekkert and Eduardo Marcos for helpful conversations in relation to \cite{BMM03}.

\section{Preliminaries}\label{sec:premilimaries}
In this section we fix  some of the notation and definitions which will be used through this paper. We fix an algebraically closed field $K$ of $\operatorname{char}\neq 2$.

\subsection{Gentle and skew-gentle algebras}

In this subsection, we define gentle and skew-gentle algebras.

A \emph{quiver} $Q$ is a quadruple $(Q_0, Q_1, s, t)$, where $Q_0$ is the set of vertices, $Q_1$ is the set of arrows and $s, t:Q_1\to Q_0$ are functions indicating the source and target of an arrow.  A \emph{path} $w$ of length $n>0$ in $Q$ is a sequence of arrows  $\alpha_1 \dots \alpha_n$ such that $t(\alpha_j)=s(\alpha_{j+1})$ for each $j=1, \dots, n-1$. For each vertex $i$, we denote by $e_i$ the trivial path of length $0$. 

The \emph{path algebra} $KQ$ is defined as the $K$-vector space with basis the set of all paths in $Q$, with multiplication induced by concatenation of paths.  A 2-sided ideal $I$ of $KQ$ is \emph{admissible} if there exists an integer $m\geq 2$ such that $R^m \subset I \subset R^{2}$, where $R$ is the ideal of $KQ$ generated by the arrows of $Q$.

Skew-gentle algebras were introduced in \cite{GdP99}. They are closely linked to the well-studied class of gentle algebras. For instance, they are skew-group algebras of gentle algebras. They are of tame representation type \cite{GdP99} and their derived categories are also tame \cite{BMM03, BD17}. In \cite{BMM03} a combinatorial description of the indecomposable objects in the bounded derived category of a skew-gentle algebra is given in terms of homotopy strings and bands.

In order to define skew-gentle algebras, we first recall the definition of gentle algebras.

\begin{defn}\label{defn:gentle}
A $K$-algebra $\Lambda$ is \emph{gentle} if it is Morita equivalent to $KQ/I$, where
\begin{enumerate}
\item\label{item:def-gentle-at-most-two} $Q$ is a finite quiver such that for every vertex $i$ of $Q$ there are at most two arrows ending at $i$ and at most two arrow starting at $i$;
\item\label{item:def-gentle-at-most-one-in-I} for every arrow $\alpha$ of $Q$, there is at most one arrow $\beta$ such that $t(\alpha) = s(\beta)$ and $\alpha\beta\in I$, and there is at most one arrow $\gamma$ such that  $t(\gamma) = s(\alpha)$ and $\gamma\alpha \in I$;
\item\label{item:def-gentle-at-most-one-outside-I} for every arrow $\alpha$ of $Q$, there is at most one arrow $\beta'$ such that $\alpha\beta'\notin I$, and there is at most one arrow $\gamma'$ such that $\gamma'\alpha\not\in I$;
\item\label{item:def-gentle-I-generated-by-paths-of-length-2} $I$ is the 2-sided ideal of $Q$ generated by certain paths of length 2;
\item\label{item:def-gentle-I-is-admissible} $I$ is an admissible ideal of $Q$.
\end{enumerate}

If $I$ satisfies \eqref{item:def-gentle-at-most-two}, \eqref{item:def-gentle-at-most-one-in-I}, \eqref{item:def-gentle-at-most-one-outside-I} and \eqref{item:def-gentle-I-generated-by-paths-of-length-2}, then we say that the quotient $KQ/I$ is \emph{locally gentle.}
\end{defn}

\begin{defn}\label{defn:skew-gentle}
A $K$-algebra $A$ is \emph{(locally)  skew-gentle} if it is Morita equivalent to an algebra $KQ/I$ where 
\begin{enumerate}
    \item $Q_1 = Q'_1 \cup S$, where for $\varepsilon \in S$, $s(\varepsilon) = t(\varepsilon)$,
    \item $I = \langle I' \cup \{ \varepsilon^2 - \varepsilon \mid \varepsilon \in S\} \rangle$,
    \item $KQ' / I'$ is a (locally) gentle algebra where $Q' = (Q'_1, Q_0)$,
    \item if $\varepsilon \in S$ then the  vertex $i = s(\varepsilon)$ is the start or the end of exactly one arrow in $Q'_1$ and if  there is an arrow $\alpha \in Q'_1$ with $t(\alpha) = i$ and an arrow $\beta \in Q'_1$ with $s(\beta) = i$ then $\alpha \beta \in I'$. Moreover, there is no other element in $S$ starting at i.
\end{enumerate}
\end{defn}

We call a vertex $i \in Q_0$ \emph{special} if there exists $\varepsilon \in S$ such that  $i = s(\varepsilon)$. We denote the set of special vertices by $Sp$. If $KQ/I$ is a skew-gentle algebra as above, we call $(Q', I', Sp)$ 
a \emph{skew-gentle triple}.

\begin{rem}
If $KQ/I$ is a skew-gentle algebra with non-empty set of special vertices, then the ideal $I$ is not admissible.
\end{rem}

An admissible presentation $KQ^{sg} /  I^{sg}$ of a skew-gentle algebra $A$ is given as follows. Let $\Lambda=KQ/I$ be the gentle algebra obtained from $A$ by deleting the special loops.

Set $$Q^{sg}_0(i) = \left\{ 
\begin{array}{ll}
\{i^+, i^-\} & \mbox{if $i \in Sp$} \\
\{ i \} & \mbox{otherwise} \\
\end{array}\right.
$$
Define 
$$Q^{sg}_0:= \bigcup_{i\in Q_0}Q^{sg}_0(i),$$
The arrows of $Q_1^{sg}$ are defined as follows. The set $Q^{sg}_1[i, j]$ of arrows from vertex $i$ to vertex $j$ is given by 
$$Q^{sg}_1[i,j]:=\{(i, \alpha,j) \mid \alpha\in Q_1, i \in Q_0^{sg}(s(\alpha)), j \in Q_0^{sg}(t(\alpha))\}.$$
The ideal $I^{sg}$ is defined as follows.

$$I^{sg}:=\Big\langle \sum_{j\in Q_0^{sg}(s(\beta))}\lambda_j(i,\alpha, j)(j,\beta, k)\mid \alpha\beta\in I, i\in Q_0^{sg}(s(\alpha)), k\in Q_0^{sg}(t(\beta)) \Big\rangle,$$
where $\lambda_j=-1$ if $j=l^-$ for some $l\in Q_0$, and $\lambda_j=1$ otherwise.

Note that in general, the relations in $I^{sg}$ are not monomial, instead, the ideal $I^{sg}$ is admissible and quadratic.

\begin{exmp}\label{ex:D_and_tildeD_are_skew-gentle}
Consider the following quivers 
$$
Q:\,  \xymatrix{4\ar[r]^{\alpha_1} & 3 \ar[r]^{\alpha_2} & 2 \ar[r]^{\alpha_3} & 1  \ar@(ur,dr)[]^{\varepsilon}}  \ \ \ \text{and} \ \ \ 
Q':\, \xymatrix{3 \ar[r]^{\alpha_1}\ar@(ul,ur)[]^{\varepsilon_1} & 2 \ar@(ul,ur)[]^{\varepsilon_2}\ar[r]^{\alpha_2} & 1,}
$$
and set $I=\{\varepsilon^2-\varepsilon\}$ and  $I'=\{\varepsilon_1^2-\varepsilon_1,\varepsilon_2^2-\varepsilon_2, \alpha_1\alpha_2\}$. Then the algebras $A_1=KQ/I$ and $A_2=KQ'/I'$ are skew-gentle and their respective admissible presentations $KQ^{sg}/I^{sg}$ and $KQ'^{sg}/I'^{sg}$ are as follows. 
$$
\xymatrix@C+1pc{ & & & 1^- \\
 Q^{sg}: \, 4 \ar[r]^{(4,\alpha_1, 3)} & 3 \ar[r]^{(3,\alpha_2, 2)} & 2 \ar[ur]^{(2, \alpha_3, 1^-)} \ar[dr]_{(2, \alpha_3, 1^+)} & \\
 & & & 1^+} \ \ \ \ \  
\xymatrix@C+1pc{3^- \ar[r]^{(3^-, \alpha_1, 2^-)} \ar[ddr]^(0.8){(3^-, \alpha_1, 2^+)} & 2^- \ar[dr]^{(2^-, \alpha_2, 1)} & \\
Q'^{sg}:\hspace{1.8cm} & & 1 \\
3^+ \ar[r]_{(3^+, \alpha_1, 2^+)} \ar[uur]^(0.3){(3^+, \alpha_1, 2^-)}  & 2^+ \ar[ur]_{(2^+, \alpha_2, 1)}& }
$$
where $I^{sg}$ is the empty set and $$I'^{sg}=\langle (3^{-}, \alpha_1, 2^+)(2^+, \alpha_2, 1)- (3^{-}, \alpha_1, 2^-)(2^-, \alpha_2, 1), \,  (3^{+}, \alpha_1, 2^+)(2^+, \alpha_2, 1)- (3^{+}, \alpha_1, 2^-)(2^-, \alpha_2, 1) \rangle.$$ 

Note that the algebra $KQ^{sg}/I^{sg} = KQ^{sg}$ corresponds to an orientation of the Dynkin diagram $\mathbb{D}_5$.
\end{exmp}

\subsection{Ribbon graphs and ribbon surfaces of a gentle algebra}

 We now briefly recall the construction of the ribbon graph of a gentle algebra embedded in a surface with boundary  as introduced in  \cite{Sch18, OPS18} based on  \cite{Sch15}.

A \emph{graph} $\gra$ is a quadruple $\gra=(M, E, s, \iota)$, where $M$ is a finite set of vertices, $E$ a finite set of half-edges, $s:E \to M$ is a function sending each half edge to the vertex it is attached to, and $\iota:E \to E$ is a fixed point free involution sending each  each half-edge to the other half-edge it is glued to.

A \emph{ribbon graph} is a graph $\gra$ endowed with a cyclic permutation  of the half-edges at each vertex, given by a function $\sigma: E\to E$ whose orbits correspond to the sets $s^{-1}(m)$, for all $m\in M$. A \emph{marked ribbon graph} is a ribbon graph $\gra$ equipped with a map $p : M \to E$, that is at each vertex we chose exactly one half-edge. 

Let $\Lambda$ be a gentle algebra. Following \cite{Sch15, Sch18, OPS18} we construct a marked ribbon graph $\gra_\Lambda$ canonically embedded into an unique, up to isomorphism, compact oriented surface $S_\Lambda$  in such a way that the faces of $\gra_\Lambda$ are in bijection with the boundary components of $S_\Lambda$.   Furthermore, the information given by $m$  translates into a gluing of the vertices of $\gra_\Lambda$ onto the boundary components of $S_\Lambda$. More precisely,

\begin{defn}\label{defn:ribbon-surface}
The \emph{marked ribbon graph} $\gra_\Lambda$ of  a gentle algebra $\Lambda=KQ/I$  is defined as follows.
The set of vertices $M$ of $\gra_\Lambda$ are in bijection with the set consisting of 
\begin{itemize}
\item \emph{maximal paths} in $KQ/I$, that is, paths $w \in KQ$, with  $w\notin I$ such that for any arrow $\alpha \in Q_1$, $\alpha w \in I$ and $w\alpha \in I$;
\item {trivial paths} $e_i$ such that $i$ is either  the source or the target of only one arrow, or $i$ is the target of exactly one arrow $\alpha$ and the source of exactly one arrow $\beta$, and $\alpha\beta\notin I$;
\end{itemize}

 The edges of $\gra_\Lambda$ are in bijection with the vertices of $Q_0$: It follows from the definition
  of $M$ that any vertex $i \in Q_0$ is in exactly two elements of $M$, that is, there exist $w_1, w_2$ in 
  $M$ such that $w_k = p_k e_i q_k$ for $p_k, q_k$ possibly trivial paths in $Q$ and $ k = 1,2$.   Hence by construction
  every vertex in $Q_0$ corresponds to exactly two elements in $M$, thus defining an edge in $
  \gra_\Lambda$. Unless otherwise specified and if no confusion arises,  we will denote the edge in $\gra_\Lambda$ corresponding to the vertex $i \in Q_0$ 
  again by $i$. 

Note that the construction of $\gra_\Lambda$  naturally gives a linear order of the half-edges attached to every vertex $w \in M$: namely,  if $i_1, \ldots, i_n$
are the half-edges at $w$ then $w = e_{i_1} \alpha_1 e_{i_2} \alpha_2 \cdots \alpha_{n-1} e_{i_n}$, with $\alpha_j \in Q_1$ and $i_j \in Q_0$ and the induced linear order of the half-edges is given by $i_1 < i_2 < \cdots < i_n$. The cyclic order $\sigma$ of the half-edges   at $w$ is then given by the cyclic closure of this linear order. 

Furthermore, we define the marking map $m: M \to E$ by $m(w)=e_{e(w)}$ for $w \in M$.
\end{defn}

We recall from \cite{OPS18} the following definition. 

\begin{defn}
The \emph{ribbon surface} $S_\Lambda$ of a gentle algebra $\Lambda$ is a tuple $S_\Lambda=(S_\Lambda, M_\Lambda)$, where $S_\Lambda$ is a compact oriented surface and a finite set $M_\Lambda$ of marked points in the boundary of $S_\Lambda$ such that $\gra_\Lambda$ is canonically embedded into  $S_\Lambda$ with faces of  $\gra_\Lambda$ corresponding to boundary components in $S_\Lambda$ and where  $M_\Lambda$ corresponds to the vertices $(\gra_\Lambda)_0$ of $\gra_\Lambda$  such that for each vertex $v\in V$, the boundary component lies between $m(v)$ and $\sigma(m(v))$ in the orientation of the surface.  
\end{defn}

\begin{rem}\label{polygons-gentle}
 By \cite[Proposition 1.12]{OPS18}, the ribbon graph $\gra_\Lambda$ of gentle algebra $\Lambda$ divides $S_\Lambda$ into polygons  of the following type
    \begin{enumerate}
        \item polygons whose edges are edges of $\gra_\Lambda$ except for exactly one boundary edge, and whose interior contains no boundary component of $S_\Lambda$;
        \item polygons whose edges are edges of $\gra_\Lambda$ and whose interior contains exactly one boundary component of $S_\Lambda$ with no marked points.
    \end{enumerate}
\end{rem}

We refer to  $(S_\Lambda,M_\Lambda, \gra_{\Lambda})$ as the \emph{surface dissection of $S_\Lambda$} associated to the gentle algebra $\Lambda$. 
In the following we will replace any boundary component with no marked points, such as those in Remark \ref{polygons-gentle} (2), with punctures. Furthermore, by abuse of notation we will sometimes treat these punctures as marked points in the interior of the surface. 

\section{Skew-gentle algebras and orbifolds}\label{sec:geometricmodel}

In this section, we define a  graph associated to a skew-gentle algebra generalising the ribbon graph associated to a gentle algebra, and we will show that this graph also has a canonical embedding into a surface.

\begin{lem}\label{lem:special-edges}
Let  $A$ be a skew-gentle algebra, $\Lambda=KQ/I$ be the associated gentle algebra obtained from $A$ by deleting all special loops and $(S_\Lambda,M_\Lambda, \gra_{\Lambda})$ be its surface dissection. Then any special vertex $i\in Sp$ corresponds to an edge $g$ in $\gra_{\Lambda}$ of a digon with exactly one boundary edge and whose interior contains no boundary component of $S_\Lambda$.
\end{lem}

\begin{proof}
Let $i$ be a special vertex of $A$ corresponding the edge $i$  in $\gra_\Lambda$ and let $P$ and $P'$ be the polygons sharing the edge $g$ . By Definition \ref{defn:skew-gentle}, the vertex $i$ is either the start or the end of exactly one arrow in $Q$ or there are exactly two arrows $\alpha, \beta$ in $Q$   such that $s(\beta)=t(\alpha)=i$ and $\alpha\beta\in I$.  

Suppose first that we are in the latter case, that is that there are exactly two arrows $\alpha$ and $\beta$  such that $s(\beta)=t(\alpha)=i$. Assume further that $s(\alpha) $ and $t(\beta)$ are edges in $P$. Then the edge $j$ of $P'$ preceding $i$ and the edge $k$ following $i$ in the orientation of  $S_\Lambda$ are boundary edges. By \cite[Proposition 1.12]{OPS18} every polygon in the surface dissection of $S_\Lambda$ has at most one boundary edge which implies that $j=k$ and $P'$ is a digon. 

In the case that there is exactly one arrow incident with $i$, the argument is similar. 
\end{proof}

We now give the definition of a generalised ribbon graph associated to a skew-gentle algebra based on the ribbon graph and surface of the underlying gentle algebra.

Let  $A$ be a skew-gentle, $\Lambda=KQ/I$ be the associated gentle algebra obtained from $A$ by deleting all special vertices, and let $(S_\Lambda, M_\Lambda, \gra_\Lambda)$ its surface dissection. For each edge $g$ of $\gra_{\Lambda}$ corresponding to a special vertex $i\in Sp$, let $P$ and $P'$ be the polygons sharing the edge $g$. Suppose further that  $P'$ is a digon with one boundary edge (which exists by the specialeness of $v$, see Lemma \ref{lem:special-edges}). We define the \emph{local replacement of $g$ in $\gra_\Lambda$} as the graph embedded graph $\gra'_{\Lambda}$ obtained by contracting the boundary segment of $P'$ and identifying the vertices $p_1$ and $p_2$ by  collapsing the interior of the polygon $P'$ so that $g$ is incident with $p_1 = p_2$. In the process we obtain a new marked point in the interior of $P$, we will denote this point by  $o$ and  depict by drawing a cross-shaped vertex in the surface. We will refer to these vertices as special vertices. We locally illustrate the local replacement and the resulting new vertex in  Figure~\ref{fig:local-replacement}.

 \begin{figure}[ht!]
        \centering
        \begin{tikzpicture}
        \filldraw (0,0) circle (1pt) node[above] {\tiny$p_1$};
        \filldraw (2,0) circle (1pt) node[above] {\tiny$p_2$};
        \filldraw (2.2,-1) circle (1pt);
        \filldraw (-0.4,-1) circle (1pt);
        \draw (0,0) ..controls (1,-0.5) ..(2,0);
        \draw[dashed] (-0.7,0.06) ..controls (1, -0.1).. (2.7,0.06);
        \draw (0,0) to[bend left] (-0.4,-1);
        \draw (-0.4,-1) to[bend left] (0.3,-1.7);
        \draw (2,0) to[bend right] (2.2, -1);
        \draw (2.2,-1) to[bend right] (1.5,-1.7);
        \node at (0.7,-1) {\tiny$P$};
        \node at (1,-0.55) {\tiny{$g$}};
        \draw (0,0) to[bend left] (-1,-0.3);
        \draw (2,0) to[bend right] (3,-0.3);
        \node at (-0.5,-0.5) {\tiny{$\dots$}};
        \node at (2.3,-0.5) {\tiny{$\dots$}};
        \node at (1,-1.7) {\tiny$\dots$};
        \node at (-1.5,0) {$\gra_\Lambda:$};
        %%%%%%%SECOND FIGURE
        \filldraw (7,0) circle (1pt) node[above] {\tiny$p_1=p_2$};
        \filldraw (6,-0.5) circle (1pt);
        \filldraw (8,-0.5) circle (1pt);
        \draw[dashed] (5.5,0.45) to [bend right] (8.5,0.45);
        \draw (7,0) to[bend left] (5.5,0.2);
        \draw (7,0) to[bend right] (8.5,0.2);
        \draw (7,0) to[bend left] (6,-0.5);
        \draw (7,0) to[bend right] (8,-0.5);
        \draw (6,-0.5) to[bend left] (6.3, -1.3);
        \draw (8,-0.5) to[bend right] (7.7,-1.3);
        \node at (7,-1.5) {$\dots$};
        \draw (7,0) to (7,-1);
        \node at (7,-1) {$\times$};
        \node at (7.1, -1.2) {\tiny$o$};
        \node at (6.5,-0.7) {\tiny$P$};
        \node at (7.1,-0.6) {\tiny$g$};
        \node at (6.2,-0.2) {\tiny$\vdots$};
        \node at (7.8,-0.2) {\tiny$\vdots$};
        \node at (4.5,0) {$\gra'_\Lambda:$};
        \end{tikzpicture}
\caption{A local replacement}
\label{fig:local-replacement}
\end{figure}
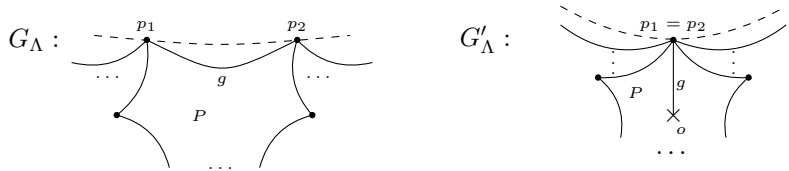

\begin{rem}\label{Rem:lamination}
(1) 
Observe that after local replacement the polygon $P$ is no longer a polygon, but corresponds to a degenerate or self-folded polygon with a special edge ending in a cross-shaped vertex.

(2) We note that up to homeomorphism the surface does not change under local replacement. However,  the number of marked points in the boundary changes.  
\end{rem}

\begin{defn}\label{generalised ribbon surface}
Let  $A$ be a skew-gentle algebra, $\Lambda=KQ/I$ be the associated gentle algebra obtained from $A$ by deleting all special vertices, and $(S_\Lambda, M_{\Lambda}, \gra_{\Lambda})$ be its surface dissection.

The \emph{generalised ribbon graph} $\gra_A$ of $A$ is the graph obtained from $\gra_\Lambda$ by applying a local replacement at each  each special vertex. 

Denote by $M_A$ the marked points on $S_\Lambda$ corresponding to the vertices of the embedded ribbon graph $G_A$ which are not special. Let $O_A=(S, M_A, \orbi)$ be the triple given by the surface $S = S_\Lambda$, the marked points $M_A$ and the set of special vertices $\orbi$. 
\end{defn}

From now on consider the special vertices of $\gra_A$ to be orbifold points of order two  and we say that an edge in $\gra_A$ joining a vertex and an orbifold point is a \emph{special edge}. Consequently, we refer to $O_A$ as the \emph{orbifold} of $A$.

Note that the above construction also works for locally skew-gentle algebras $A  = KQ/I$. In this case the generalised ribbon graph will have in addition to the special vertices punctures corresponding to either cycles with no relations in the quiver  or to cycles with full relations, but with a special loop at each vertex of the cycle.

\begin{rem}\label{rem:GeometricRibbonGraph}
We note that we can construct the generalised ribbon graph directly from the data of  $(Q, I , Sp)$,  where $A = KQ/I$ is a skew-gentle algebra with set of special vertices $Sp$.  We say that a path $p$ in $Q$ is $Sp$-\emph{maximal} if for all $x \in Q_1\setminus Sp$ we have $px = xp = 0$ in $KQ/I$. 
Then  the set of vertices $M $ of the ribbon graph $G_A$ of $A$ is in bijection with the union of all
\begin{itemize}
\item  $Sp$-maximal paths; 
\item trivial paths $e_i$ such that $i$ is either the source or the target of only one arrow, or $i$ is the target of exactly one arrow $\alpha$ and the source of exactly one arrow $\beta$, and $\alpha \beta \notin I$, or $i \in Sp$. 
\end{itemize}
The set of edges of $G_A$ is in bijection with the vertices of $Q_0$ (note that this includes the special vertices). Then $G_A$ is a ribbon graph with the cyclic ordering of the edges at each vertex induced by the $Sp$-maximal paths. Denote by $S_A$ the corresponding oriented surface with boundary such that $G_A$ is a deformation retract of $S_\Lambda$. Now define a marking 
map $m : M\setminus Sp \to E$ similar to the gentle case. Note that the elements of $M$ corresponding to the idempotents at special vertices in $Q_0$  are not marked. 
The marking map gives a unique way of gluing the marked vertices of $G_A$ to the boundary of $S_A$ whereas the vertices of $G_A$  corresponding to elements in $Sp$ stay in the interior of the surface where they give rise to the set $\mathcal O$ of orbifold points of order two. 
This gives rise to a generalised surface dissection which coincides with the construction of the orbifold dissection of $O_A$ in Definition~\ref{Def:orbifold-dissection} below. %\yad{$G_A$ is a generalised ribbon graph or a ribbon graph? should we have two different kind of vertices in $G_A$?} \sib{We do have to kinds of vertices in $G_A$, namely those corresponding the special vertices in $Q_0$ and all the others, I have add a little bit more detail on this. }
\end{rem}

\begin{exmp}\label{ex:D_and_tildeD_ribbonGraphsAndSurfaces}
Let $A_1=KQ/I$ and $A_2=KQ'/I'$ be the skew-gentle algebras from Example \ref{ex:D_and_tildeD_are_skew-gentle}. Note that the set of vertices ${M}_{A_1}$ of the generalised ribbon graph $G_{A_1}$ is the set $\{\alpha_1\alpha_2\alpha_3\varepsilon, e_4, e_3, e_2, e_1\}$ where $e_1$ corresponds to the trivial path associated to the special vertex $1$ and that the set of vertices $M_{A_2}$ of $A_2$ is the set $\{\varepsilon_1\alpha_1\varepsilon_2\alpha_2, e_3, e_2, e_1\}$ where $e_3$ and $e_2$ correspond to the trivial path associated to the special vertices $3$ and $2$ respectively. The generalised ribbon graphs $G_{A_1}$ and $G_{A_2}$ can be seen in Figure \ref{Fig:ribbon graph D_5 and A_2}.
\begin{figure}[ht!]
\centering
\begin{tikzpicture}
\node at (0,0) {$\alpha_1\alpha_2\alpha_3\varepsilon$};
\filldraw (0,-0.2) circle (1pt);
\filldraw (216:1.3cm) circle (1pt);
\filldraw (252:1.3cm) circle (1pt);
\filldraw (288:1.3cm) circle (1pt);
%\filldraw (322:1.3cm) circle (1pt);
\node at (322:1.3cm) {\tiny$\times$};
\node at (216:1.5cm) {$e_4$};
\node at (252:1.5cm) {$e_3$};
\node at (288:1.5cm) {$e_2$};
\node at (322:1.6cm) {$e_1$};
\draw (0,-0.2) to (216:1.3cm);
\draw (0,-0.2) to (252:1.3cm);
\draw (0,-0.2) to (288:1.3cm);
\draw (0,-0.2) to (322:1.3cm);
%%% Second figure
\node at (4,0) {$\varepsilon_1\alpha_1\varepsilon_2\alpha_2$};
\filldraw (4,-0.2) circle (1pt);
\draw (225:1.3cm) ++(4cm,0pt)circle (0.00001cm) coordinate (a);
%\filldraw (350:1cm) ++(4cm,0pt)circle (1pt) coordinate (c);
\draw (270:1.3cm)  ++(4cm,0pt)circle (0.00001cm) coordinate (b);
\filldraw (315:1.3cm)  ++(4cm,0pt)circle (1pt) coordinate (c);
\node at (b) {\tiny$\times$};
\node at (a) {\tiny$\times$};
\node[below] at (a) {$e_3$};
\node[below] at (b) {$e_2$};
\node[below] at (c) {$e_1$};
\draw (4,-0.2) to (a);
\draw (4,-0.2) to (b);
\draw (4,-0.2) to (c);
\end{tikzpicture}
    \caption{Generalised ribbon graphs of $\mathbb D_5$ and $A_2$ from Example \ref{ex:D_and_tildeD_are_skew-gentle}}
    \label{Fig:ribbon graph D_5 and A_2}
\end{figure}
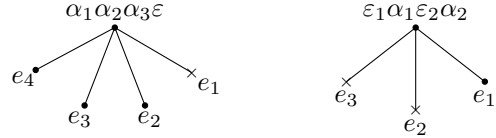

Then the generalised ribbon graphs $G_{A_1}$ and $G_{A_2}$ embedded in their respective orbifolds can be seen in Figure\ref{Fig:types_D5_and_tildeD6ribbonGraphsAndSurfaces}.

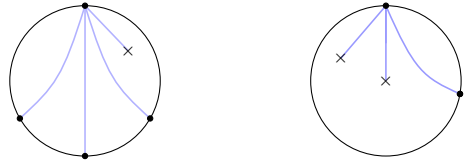
\begin{figure}[ht!]
\centering
\begin{tikzpicture}
\draw (0,0) circle (1cm);
\draw[semithick,blue!30!white] (90:1cm) to (270:1cm);
\draw[semithick,blue!30!white] (90:1cm)  ..controls (170:0.3cm) and (190:0.4cm) .. (210:1cm);
\draw[semithick,blue!30!white] (90:1cm) ..controls (10:0.3cm) and (350:0.4cm).. (330:1cm);
\draw[semithick,blue!30!white] (90:1cm) to (35:0.7cm);
\filldraw (90:1cm) circle (1pt);
\filldraw (270:1cm) circle (1pt);
\filldraw (210:1cm) circle (1pt);
\filldraw (330:1cm) circle (1pt);
\node at (35:0.7cm) {\tiny$\times$};
%%%%% SECOND FIGURE
\draw (4,0) circle (1cm);
%\filldraw (270:1cm) ++(4cm,0pt) circle (1pt) coordinate (a);
%\filldraw (210:1cm) ++(4cm,0pt) circle (1pt) coordinate (b);
\filldraw (350:1cm) ++(4cm,0pt)circle (1pt) coordinate (c);
\draw[white] (170:0.3cm) ++(4cm,0pt) circle (0.4pt) coordinate (d);
\draw[white] (190:0.4cm) ++(4cm,0pt) circle (0.4pt) coordinate (e);
\draw[white] (20:0.4cm) ++(4cm,0pt) circle (0.4pt) coordinate (f);
\draw[white] (5:0.5cm) ++(4cm,0pt) circle (0.4pt) coordinate (g);
\draw[semithick,blue!40!white] (4,1) to (4,0);
\draw[semithick,blue!40!white] (4,1) to (4,0.5);
%\draw[semithick,blue!40!white] (4,1)  ..controls (d) and (e).. (b);
\draw[semithick,blue!40!white] (4,1) ..controls (f) and (g).. (c);
%\draw[semithick,blue!40!white] (4,1) to (4.6,0.3);
\draw[semithick,blue!40!white] (4,1) to (3.4,0.3);
\filldraw (350:1cm) ++(4cm,0pt)circle (1pt) coordinate (c);
\filldraw (4,1)circle (1pt) coordinate (c);
\node at  (4,0) {\tiny$\times$};
\node at  (3.4,0.3) {\tiny$\times$};
\end{tikzpicture}
    \caption{Generalised ribbon graphs of $\mathbb D_5$ and $KQ'/ I$ from Example \ref{ex:D_and_tildeD_are_skew-gentle} embedded in their respective orbifolds.}
    \label{Fig:types_D5_and_tildeD6ribbonGraphsAndSurfaces}
\end{figure}

\end{exmp}

For any skew-gentle algebra $A$, the edges of $\gra_A$ cut the orbifold  $O_A$ into polygons, some of which contain the points in $\orbi$ and a special edge connected to them. We call those polygons \emph{degenerate polygons}. We note that the following two results, Proposition~\ref{Prop:degenerate-polygons} and Theorem~\ref{Thm:skew-gentle orbifold}, have independently appeared  in \cite{AB19}.

\begin{prop}\label{Prop:degenerate-polygons}
Let $A$ be a skew-gentle algebra, and let $\gra_A$ be the generalised ribbon graph of $A$ embedded into its orbifold $O_A=(S_A, M_A, \orbi)$. Then $\gra_A$ cuts $O_A$ into four types of polygons:

\begin{enumerate}
    \item[a)]  polygons and degenerate polygons containing exactly one boundary segment whose interior contains no boundary component of $O_A$.
    \item[b)] polygons and degenerate polygons with no boundary segments  and whose interior contains exactly one boundary component of $O_A$ with no marked points.
\end{enumerate}
\end{prop}

\begin{proof} This directly follows from \cite[Proposition 1.12]{OPS18} and the construction of $\gra_A$ by local replacement. 
\end{proof}

\begin{defn}\label{Def:orbifold-dissection}
We call an \emph{orbifold dissection} any tuple of the form $\dissection$, where $S$ is a compact oriented surface with marked points $M$, orbifold points $\orbi$ of order 2 and $\gra$ is a graph as in Remark~\ref{rem:GeometricRibbonGraph} dissecting $S$ into polygons and degenerate polygons of the form as described in Proposition~\ref{Prop:degenerate-polygons}.
\end{defn} 

Before stating the next result, we define the following notation. Denote by $B$ one of the skew-gentle algebras with two vertices, one arrow between them and one or two special loops. %We note that the next result already appears in \cite{AB19}, but for completeness we also give a proof.

\begin{thm}\label{Thm:skew-gentle orbifold}
Every skew-gentle algebra non-isomorphic to $B$  uniquely determines an orbifold dissection up to homemorphism and every orbifold dissection uniquely determines a skew-gentle algebra. 
\end{thm}

\begin{proof}
By Proposition~\ref{Prop:degenerate-polygons} it is enough to show that given a orbifold dissection $\dissection$ there exists a skew-gentle algebra $A$ having $(S,M,\orbi)$ as its orbifold and $\gra$ as its generalised ribbon graph. 
Given $(S,M, \orbi, \gra)$, define a 
quiver $Q$ as follows:
\begin{enumerate}
    \item the vertices of $Q$ are in bijection with the edges of $\gra$;
    \item if $i, j$ are two edges incident with the same vertex in $\gra$ then there is an arrow from $i$ to $j$ if $j$ is a direct successor of $i$ in the orientation of the surface, that is there is no other edge of $\gra$ between $i$ and $j$.
    Note that if $i$ is a special edge in a degenerate polygon, then $i$ is its own successor, therefore there is a loop incident to $i$.  Denote by $Sp$ the set of vertices of $Q$ which corresponds to the special edges of $\gra$ and by $S$ the set of special loops incident to a special vertex $i\in Sp$.
\end{enumerate}
Observe that by construction any vertex of $Q$ has at most two in going arrows and at most two outgoing arrows, because any edge of $\gra$ shares at most two (degenerate) polygons.
Let $I$ be the ideal of $KQ$ generated by the following relations: if $\alpha: i \to j$ and $\beta:j\to k$ are two consecutive arrows such that $i, j$ and $k$ correspond to edges of the same (degenerate) polygon and such that neither $\alpha$ nor $\beta$ correspond to a special loop, then $\alpha\beta$ is a relation. Consequently, for any arrow $\alpha$, there is at most one consecutive arrow $\beta$ such that $\alpha\beta\in I$ and at most one preceding arrow $\gamma$ such that $\gamma\alpha\in I$. Finally, for each loop $\varepsilon$, incident to a special vertex, $\varepsilon^2-\varepsilon\in I$.

We need to show that $A=KQ/I$ is a skew-gentle algebra. By construction, if $P$ is a degenerate polygon, and $v$ is a special edge of $P$, then the valency of $i$ is at most two, because any special edge belongs to exactly one degenerate polygon. Moreover, if the valency of $i$ is two, by the definition of the generators of $I$, the composition of the arrow ending at $i$ with the arrow starting at $i$ is a generator.

Finally it follows from the construction of $KQ/I$ that its orbifold dissection is $(S,M, \orbi, \gra)$.
\end{proof}

\section{The dual graph and the Koszul dual of a skew-gentle algebra}

In this subsection, we show that a skew-gentle algebra is strong Koszul and that its Koszul dual is again skew-gentle. Furthermore, given a skew-gentle algebra $A$ and the corresponding skew orbifold dissection, we construct a dual graph embedded in the orbifold and we show that this dual graph is the orbifold dissection of  the Koszul dual  of $A$.

According to \cite{Green2017}, an algebra $KQ/I$ is \emph{strong Koszul} if $I$ is quadratic and has a quadratic Gr\"obner basis. By \cite{GH95} any strong Koszul algebra is  a Koszul algebra, but the converse does not always hold, for example, Sklyanin algebras \cite{Smith94} are Koszul algebras but not strong Koszul.

Since gentle algebras are Koszul and since skew-group constructions preserve the Koszul property \cite{Lip15}, it is clear that skew-gentle algebras are Koszul. However, it is not known whether the skew-group algebra of a strong Koszul algebra is strong Koszul.  The aim of this section is prove that skew-gentle algebras are strong Koszul. As  as consequence we also give a new proof that skew-gentle algebras are Koszul.
Following \cite{Gre99} and \cite{GHS17}, we recall the basic definitions of  Gr\"obner bases. Recall that  $\succ$ is an \emph{admissible order} on $B$ if $\succ$ is a total order on $B$ such that every nonempty subset of $B$ has a minimal element, and $\succ$ is compatible with the multiplicative structure of $B$, namely the following conditions hold for any $p,q,x,w \in B$, see \cite[Section 2.2.1]{Gre99} for details.

\begin{enumerate}

\item if $p\succ q$ then $px\succ qx$ when $px\neq 0$ and $qx\neq 0$.

\item if $p\succ q$ then $wp\succ wq$ when $px\neq 0$ and $qx\neq 0$.

\item if $p=qx$, then $p\succ q$ and $p\succ x$.

\end{enumerate}

\begin{defn}
Let $Q$ be a quiver, $\mathcal B$ the basis of paths of $KQ$ and $\succ$ an admissible order on $\cB$.   For $x=\sum_{p\in \mathcal B}\lambda_p p$ with $\lambda_p\in K$ such that almost all $\lambda_p=0$ define the \emph{tip of $x$}  to be
$\operatorname{tip}(x)=p \textrm{ if } \lambda_p\neq 0 \textrm{ and } p \succ q \textrm{ for all } q \textrm{ with } \lambda_p\neq 0.$
Furthermore, if $X\subset KQ$ then we define
$\operatorname{tip}(X)=\{\operatorname{tip}(x)\mid x\in X\setminus\{0\}\}$.
\end{defn}

To simplify notation, in this section a vertex idempotent $e_i$ associated to a vertex $i$ will be denoted by $i$. An element $x\in KQ$ is \emph{uniform} if there are vertices $i, j\in Q_0$
such that $ixj = x$.

\begin{defn}
Let $KQ/I$ be an algebra and let $\succ$ be an admissible order on the basis of paths $\cB$ of $KQ$. 
We say that  $\cG \subset I$ is a \emph{Gr\"obner basis} for $I$ with respect  $\succ$ if $\cG$ is a set of
uniform elements in $I$ such that
$\langle \operatorname{tip}(I)\rangle = \langle \operatorname{tip}(\cG)\rangle$.
\end{defn}

The concepts of simple and complete reductions play an important role in Gr\"obner basis theory. For example, a useful characterisation of a Gr\"obner basis is that all its elements completely reduce to zero, see, for example,  \cite[Proposition 2.9]{GHS17}. Furthermore, complete reduction can be used to complete a subset $\cH$ of uniform elements of the ideal  $I$ to a Gr\"obner basis.  

\begin{defn}
Let $\mathcal H$ be a set of nonzero uniform elements in $KQ$ and $x =\sum_{p\in \mathcal B}\lambda_p p\neq 0$ be an element of $KQ$ with $\lambda_p \in K$.
\begin{itemize}
    \item \emph{A simple reduction of $x$ by $\mathcal H$ is defined as follows}: Suppose for some $p$ with $\lambda_p\neq 0$ there exists $h \in \mathcal H$  and $r,s  \in \cB$ such that $r \operatorname{tip}(h) s = p$. If $\lambda$ is the coefficient of $\operatorname{tip}(h)$  as a summand of  $h$  as a linear combination of basis elements  then
a simple reduction $x \to_{\mathcal H} y$ of $x$ by $\mathcal H$ is $y = \lambda x-\lambda_p rhs$. This replaces  $\lambda_p p$ in $x$ by a linear combination of paths smaller than $p$.

    \item \emph{A complete reduction $x {\implies}_{\mathcal H} y_n$ of $x$ by $\mathcal H$}
is a sequence of simple reductions $(\dots((x \to_{\mathcal H} y_1) \to_{\mathcal H} y_2) \to_{\mathcal H} \dots) \to_{\mathcal H} y_n$, such that
either $y_n = 0$ or $y_n$ has no simple reductions by H.
\end{itemize} 
\end{defn}

\begin{defn}
Let $x = \sum_{p\in \mathcal B}\lambda_pp$ and $y=\sum_{q\in\mathcal B}\mu_qq$ in $KQ$. Suppose that $s = \operatorname{tip}(x), t= \operatorname{tip}(y)$ and $sm = nt$
for some $m, n \in \mathcal B \setminus Q_0$ where the lengths of $m$ and $n$ are strictly less than the length of $s$. Then the overlap relation, $\mathbf {o}(x, y, m, n)$, is
$$\mathbf{o}(x, y, m, n) = (\mu_t)xm-(\lambda_s)ny.$$
\end{defn}

We now show that a skew-gentle algebra is strongly Koszul by showing that it has a quadratic Gr\"obner basis. For this we begin by showing that given a skew-gentle algebra $A$, any admissible order for the underlying gentle algebra $\Lambda$ induces  a natural admissible order for $A$.

\begin{lem}\label{Lemma:Admissible-order}
Let $A$ be a skew-gentle, $\Lambda=KQ/I$ be the gentle algebra obtaining from $A$ by deleting special loops, and let $\cB$ be the basis of paths of $KQ$.  Suppose that $\succ_Q$ is an admissible order on $\cB$. Then $\succ_Q$ induces an admissible order $\succ$ on the basis of paths $\cB^{sg}$ of $KQ^{sg}$.
\end{lem}

\begin{proof}
Let $\succ_{Q}$ be an  order on $Q$ inducing an admissible order on $\cB$. Consider $\succ$ an order on $Q^{sg}$ induced by $\succ_{Q}$ as follows.

Let $(i, \alpha, j)$ and $(i', \beta, j')$ be two arrows in $Q^{sg}$. If $\alpha\neq \beta$, we say that  $(i, \alpha, j)\succ (i', \beta, j')$ if and only if $\alpha \succ_Q \beta$.

Now, suppose that $\alpha=\beta$ and $s(\alpha)$ or $t(\alpha)$ is a special vertex, then we fix an order on the set of arrows induced by $\alpha:i \to j$ as follows.

If $i$ and $j$ are special vertices, then $(i^+, \alpha, j^+) \succ (i^-,\alpha, j^-) \succ (i^+, \alpha, j^-) \succ (i^-, \alpha, j^+)$. If $i$ is a special vertex and $t(\alpha)$ is not special (resp. $i$ is not special and $j$ is special), then $(i^+, \alpha, j) \succ (i^-,\alpha, j)$ (resp. $(i, \alpha, j^+) \succ (i,\alpha, j^-)$).

By construction the only property we need to check in order for $\succ$ to be admissible, is that every descending chain in $\cB^{sg}$ has a minimal element.  Let C=($p_1\succ p_2\succ \dots \succ p_t \succ \dots)$ be a descending chain  of paths in $\cB^{sg}$. By construction of $kQ^{sg}$, each path $p_r$ is a sequence of arrows of the form $(i_{r_1}, \alpha_{r_1},j_{r_1})\dots (i_{r_l}, \alpha_{r_l},j_{r_l})$ such that  $\alpha_{r_1}\dots \alpha_{r_l}$ is a path $\widehat{p_r}$ in $Q$ and  $C'=(\widehat{p_1}\succ_Q \widehat{p_2}\succ_Q \dots \succ_Q \widehat{p_t} \succ_Q \dots)$ is a descending chain of elements in $\cB$. Since $\succ_Q$ is an admissible order on $\cB$, the chain $C'$ has a minimal element $\widehat{p}_r$, for some $r\in\mathbb N$, such that  $\widehat{p}_r=\widehat{p}_{r+m}$ for all $m\in\mathbb N$, therefore there exists some $r' \in \mathbb N$ such that $p_{r'}=p_{r'+m}$ for all $m\in\mathbb N$.

Thus the order $\succ$ on the basis of paths $\cB^{sg}$ of $KQ^{sg}$ is admissible.
\end{proof}

\begin{prop}
Let $A$ be a locally skew-gentle algebra, then $A$ is strong Koszul.
\end{prop}

\begin{proof}
	Let $A$ be a locally skew-gentle algebra, $\Lambda=KQ/I$ be the locally gentle algebra obtained from $A$ by deleting special loops. Consider the admissible presentation of $A$, namely $A^{sg}=KQ^{sg}/I^{sg}$. %Recall that $Q^{sg}$ is obtained by doubling the special vertices, introducing arrows to and from these vertices corresponding to the previous such arrows, the arrows in $Q^{sg}$ are denoted by a triple $(i, \alpha, j)$ where $i\in Q_0(s(\alpha)) $ and $j\in Q_0(t(\alpha))$.
By \cite[Theorem 3]{GH95} it is enough to prove that there exists a quadratic Gr\"obner basis for the ideal $I^{sg}$. Let $\succ_Q$ be an admissible order on $\cB$ the basis of path of $KQ$, for example, and more precisely let $\succ_Q$ be a paths length lexicographical order. By Lemma \ref{Lemma:Admissible-order}, there exists an admissible order $\succ$ on $\cB^{sg}$, the basis of paths of $KQ^{sg}$ such that:
\begin{enumerate}
    \item if $\alpha\neq \beta$ then $(i, \alpha, j)\succ (i', \beta, j')$ if and only if $\alpha \succ_Q \beta$;
    \item the order on the set of of arrows of $Q^{sg}$ associated to an arrow $\alpha$ in $Q$ ending or starting in a special vertex is given as follows: If $i$ and $j$ are special vertices, then $(i^+, \alpha, j^+) \succ (i^-,\alpha, j^-) \succ (i^+, \alpha, j^-) \succ (i^-, \alpha, j^+)$. If $i$ is a special vertex and $t(\alpha)$ is not special (resp. $i$ not special and $j$ special), then $(i^+, \alpha, j) \succ (i^-,\alpha, j)$ (resp. $(i, \alpha, j^+) \succ (i,\alpha, j^-)$).
\end{enumerate}

We claim that the set
$$\mathcal G= \{\sum_{j\in Q_0^{sg}(s(\beta))}\lambda_j(i,\alpha, j)(j,\beta, k)\mid \alpha\beta\in I, i\in Q_0^{sg}(s(\alpha)), k\in Q_0^{sg}(t(\beta))\}$$
where $\lambda_j=-1$  if $j=l^-$ for some $l\in Q_0$, and $\lambda_j=1$ otherwise, is Gr\"obner basis for $I^{sg}$. By \cite[Theorem 2.13]{GHS17} it is enough to show that every overlap relation of any two elements of $\mathcal{G}$ completely reduces to $0$ by $\mathcal{G}$. It follows from the definition of $A^{sg}$ that any element in $\cG$ is a linear combination with at most two summands. 
Let $x=\lambda_1(i, \alpha, j)(j, \beta, k)+ \lambda_2(i, \alpha, j')(j', \beta,k)$ and $y=\mu_1(j,\beta,k)(k, \gamma, l) +\mu_2(j,\beta, k')(k', \gamma, l)$ be elements in $\mathcal G$ and $n, m\in \mathcal B$ such that $\mathbf o(x, y, m,n)$ is a overlap relation. Observe that in this case, by the definition of the overlap relation, $n$ and $m$ are arrows.
It is easy to check that if $x$ or $y$ are monomial relations, then $\mathbf o(x, y, m,n)$ is also a monomial relation in $A^{sg}$. Suppose that $x$ and $y$ are binomial relations, and 
suppose $t=\operatorname{tip}(x)=(i, \alpha, j)(j, \beta, k)$ and $t'=\operatorname{tip}(y)=(j,\beta,k)(k, \gamma, l)$, then the overlap relation is written as follows
 \begin{equation*}
  \begin{aligned}
   \mathbf{o}(x,y,n,m) &=\mu_1\lambda_2((i, \alpha, j')(j', \beta,k))(k, \gamma, l)-\lambda_1\mu_2 (i, \alpha, j)((j,\beta, k')(k', \gamma, l)) \\
&= -((i, \alpha, j')(j', \beta,k))(k, \gamma, l)+ (i, \alpha, j)((j,\beta, k')(k', \gamma, l)),
\end{aligned}
\end{equation*}
where $n=(k, \gamma, l)$ and $m=(i, \alpha, j)$. Observe that by the definition of the relations in $I^{sg}$ and by definition of  the admissible order $\succ$, we have that  $(j,\beta,k)$ is either $(s(\beta)^+, \beta, t(\beta)^+)$ or $(s(\beta)^{-}, \beta, t(\beta)^-)$. Moreover, there is no element $z\in \mathcal{G}$ such that $\operatorname{tip}(z)$ is starting with $(j', \beta, k)$ or $(j, \beta, k')$. Therefore any simple reduction  of $\mathbf{o}(x,y,n,m)$ replaces the second element $(i, \alpha, j)(j,\beta, k')(k', \gamma, l)$ with:
$$-((i, \alpha, j')(j', \beta,k))(k, \gamma, l)-(i, \alpha, j')(j',\beta, k')(k', \gamma, l).$$
Finally, if there exist an element $w$ in $\mathcal{G}$ such that $\operatorname{tip}(w)$ is starting with $(j', \beta, k')$  then we can reduce $(i, \alpha, j')(j',\beta, k')(k', \gamma, l)$ as follows:
$$-((i, \alpha, j')(j', \beta,k))(k, \gamma, l)-(i, \alpha, j')(j',\beta, k)(k, \gamma, l)=0.$$
Which implies that the quadratic basis $\mathcal G$  is a Gr\"obner basis, and as a consequence $A$ is a Koszul algebra.
\end{proof}

Before we state the last result of this section, we recall the definition of Koszul dual, see \cite{MV07} for details. The Koszul dual $A^!$ of a finite dimensional Koszul algebra $A$ is by definition the algebra $\operatorname{Ext}_A(A/ \operatorname{rad}(A), A/ \operatorname{rad}(A))^{op}$, which is isomorphic to the quadratic dual of $A$.  For the convenience of the reader we briefly recall the construction of $A^!$ for algebras of the form $KQ/I$. Let $V=KQ_2$ be the vector space generated by the paths of length two and $\{\gamma_1, \dots, \gamma_r\}$ be a basis of $V$. Denote by $V^{op}$ the vector space generated by paths of length $two$ in $Q^{op}$ with dual basis $\{\gamma_1^{op}, \dots, \gamma_n^{op}\}$.

Following \cite{MV07}, the orthogonal ideal $I_2^{\bot}$ is generated  by  $$ B=\{v\in V^{op}\mid \langle u,v\rangle=0 \text{\ for every \ } u\in I \},$$ where $\langle \, , \rangle:V \times V^{op}\to k $ is a bilinear form
defined on bases elements as follows:

$$\langle \gamma_i, \gamma_j^{op}\rangle=\begin{cases}
0 & \text{\ if \ } \gamma_i\neq \gamma_j,\\
1 & \text{\ otherwise. }
\end{cases}$$

Then the Koszul dual $A^!$ of $A$ is the path algebra $KQ^{op}/ I_{2}^{\bot}$.

\begin{prop}\label{prop:dual-algebra}
Let $A=KQ^{sg}/I^{sg}$ be the admissible presentation of a skew-gentle algebra  and $A^!$ its Koszul dual. Then the admissible ideal $I_2^{\bot}$ of $A^{!}$ is generated by:
\begin{itemize}
    \item paths of length two which are not a summand of a minimal generator of   $I^{sg}$ 
    \item commutativity relations in $I^{sg}$.
\end{itemize}
\end{prop}

\begin{proof}
Let $W$  be the set of generators of the ideal $I^{sg}$, namely
$$W= \{\sum_{j\in Q_0^{sg}(s(\beta))}\lambda_j(i,\alpha, j)(j,\beta, k)\mid \alpha\beta\in I, i\in Q_0^{sg}(s(\alpha)), k\in Q_0^{sg}(t(\beta))\},$$
where $\lambda_j=-1$  if $j=l^{-}$ for some $l\in Q_0$, and $\lambda_j=1$ otherwise.

Then, the orthogonal ideal $I_2^{\bot}$ is generated  by  $$ B=\{v\in V^{op}\mid \langle u,v\rangle=0 \text{\ for every \ } u\in W \},$$ where $\langle \, , \rangle:V \times V^{op}\to k $ is a bilinear form
defined on bases elements as follows:

$$\langle \gamma_i, \gamma_j^{op}\rangle=\begin{cases}
0 & \text{\ if \ } \gamma_i\neq \gamma_j,\\
1 & \text{\ otherwise. }
\end{cases}$$

By definition of $B$, any path of length two which is not a summand of a minimal generator of $I^{sg}$ is an element of $B$. Let $x$ be uniform element in $B$ which is by definition of $V$ a linear combination of at most two paths. To fix notation, let $x=\lambda_1\rho_1+\lambda_2\rho_2$ be a linear combination of two paths $\rho_1$ and $\rho_2$ with the same source and target with $\lambda_1, \lambda_2 \in K^*$.

%We claim that $\rho_1$ and $\rho_2$ appear in the same commutativity relation. 
Suppose for contradiction that $\rho_1$ is a monomial relation. Then $\langle \rho_1, x\rangle=0$ implies that $\lambda_1=0$ which is a contradiction. Therefore without loss of generality, we have $\rho_1 - \rho_2 \in V$ and  $\langle \rho_1-\rho_2, x\rangle=0$ implies that $\lambda_1=\lambda_2$ that is $x = \rho_1+\rho_2$.

Using the isomorphism $\varphi:KQ^{sg} \to KQ^{sg}$ defined by

$$\varphi((i, \alpha, j))=\begin{cases}
-(i, \alpha, j) & \text{\, if\, } i=s(\alpha)^+ \text{\, and \, } j=t(\alpha)^+ \text{\, or \, }  i=s(\alpha)^- \text{\, and \, } j=t(\alpha)^-;\\
(i, \alpha, j) & \text{\, otherwise. \,}
\end{cases}$$
the result follows.
\end{proof}

The following Theorem shows how to compute the generalised ribbon graph of the Koszul dual of a skew-gentle algebra. 

\begin{thm-defn}\label{Thm:dual graph} 
Let $A$ be a skew-gentle algebra, and let $(S,M,\orbi, G_A)$ be the orbifold dissection of $A$.
Denote by $\gra_A^*$ the graph embedded in a surface obtained from $(S, M, \orbi)$ as follows.
\begin{itemize}
    \item In each boundary edge of the dissection  $\gra_A$, there is exactly one vertex of  $\gra_A^*$.
    In addition, any unmarked boundary in $(S,M,\orbi)$ is replaced by a vertex of $\gra_A^*$.
    \item Any orbifold point in  $\gra_A$ is also an orbifold point in $\gra_A^*$.
\end{itemize}

Then for every non-special edge $v$ of $\gra_A$ there is a unique edge in $\gra_A^*$ crossing $v$ exactly once.  Every special edge  of $\gra_A$ corresponds to an edge of $\gra_A^*$ connecting the orbifold point with the  unique vertex of $\gra_A^*$ such that the resulting edge does not cross any edge of $\gra_A$. We call these edges of $\gra_A^*$ special edges.

Then $\gra_A^*$ is the graph of the Koszul dual $A^{!}$.
\end{thm-defn}

\begin{proof}
This follows directly from the construction of $\gra_A^*$ and Proposition~\ref{prop:dual-algebra}.
\end{proof}

In Figure \ref{Fig:types_D5_and_tildeD6_dualGraphs} we can see the dual graphs of the generalised ribbon graphs from example \ref{ex:D_and_tildeD_ribbonGraphsAndSurfaces}.
\begin{figure}[ht!]
    \centering
\begin{tikzpicture}[scale=0.8]
\draw (0,0) circle (1.5cm);
\draw[semithick,blue!20!white] (90:1.5cm) to (270:1.5cm);
\draw[semithick,blue!30!white] (90:1.5cm)  ..controls (170:0.45cm) and (190:0.6cm) .. (210:1.5cm);
\draw[semithick,blue!30!white] (90:1.5cm) ..controls (10:0.45cm) and (350:0.6cm).. (330:1.5cm);
\draw[semithick,blue!20!white] (90:1.5cm) to (35:1.05cm);
\draw[semithick, red] (180:1.5cm) to[bend left] (240:1.5cm);
\draw[semithick, red] (240:1.5cm) to [bend left](300:1.5cm);
\draw[semithick, red] (300:1.5cm) to  [bend left] (0:1.5cm);
\draw[semithick,red] (0:1.5cm) to (35:1.05cm);
\filldraw (90:1.5cm) circle (1pt);
\filldraw (270:1.5cm) circle (1pt);
\filldraw (210:1.5cm) circle (1pt);
\filldraw (330:1.5cm) circle (1pt);
\filldraw[red] (180:1.5cm) circle (1pt);
\filldraw[red] (240:1.5cm) circle (1pt);
\filldraw[red] (300:1.5cm) circle (1pt);
\filldraw[red] (0:1.5cm) circle (1pt);
\filldraw[red] (35:1.05cm) circle(0.02cm);
\node at (35:1.05cm) {\tiny$\times$};
%%%%% SECOND FIGURE
\draw (4,0) circle (1.5cm);
\filldraw (350:1.5cm) ++(4cm,0pt)circle (1pt) coordinate (c);
\draw[white] (170:0.45cm) ++(4cm,0pt) circle (0.4pt) coordinate (d);
\draw[white] (190:0.6cm) ++(4cm,0pt) circle (0.4pt) coordinate (e);
\draw[white] (20:0.6cm) ++(4cm,0pt) circle (0.4pt) coordinate (f);
\draw[white] (5:0.75cm) ++(4cm,0pt) circle (0.4pt) coordinate (g);
\draw[semithick,blue!30!white] (4,1.5) to (4,0);
\draw[semithick,blue!30!white] (4,1.5) to (4,0.5);
\draw[semithick,blue!30!white] (4,1.5) ..controls (f) and (g).. (c);
\draw[semithick,blue!30!white] (4,1.5) to (3.4,0.3);
\filldraw (350:1.5cm) ++(4cm,0pt)circle (1pt) coordinate (c);
\filldraw (4,1.5)circle (1pt) coordinate (c);
\node at  (4,0) {\tiny$\times$};
\node at  (3.43,0.35) {\tiny$\times$};
\filldraw[red] (270:1.5cm) ++(4cm,0pt) circle (1pt) coordinate (a);
\filldraw[red] (45:1.5cm) ++(4cm,0pt) circle (1pt) coordinate (b);
\draw[semithick, red] (4,-1.5cm) to (4,0);
\draw[semithick, red] (4,-1.5) to (3.43, 0.35);
\draw[semithick, red] (4,-1.5) ..controls (4.24,-0.2) and (4.85, 0.75) ..(b);
%%%%%
%
%
%\draw (4,0) circle (1.5cm);
%\filldraw (270:1.5cm) ++(4cm,0pt) circle (1pt) coordinate (a);
%\filldraw (210:1.5cm) ++(4cm,0pt) circle (1pt) coordinate (b);
%\filldraw (330:1.5cm) ++(4cm,0pt)circle (1pt) coordinate (c);
%%
%\draw[white] (170:0.3cm) ++(4cm,0pt) circle (0.4pt) coordinate (d1);
%\draw[white] (190:0.4cm) ++(4cm,0pt) circle (0.4pt) coordinate (e1);
%\draw[white] (10:0.3cm) ++(4cm,0pt) circle (0.4pt) coordinate (f1);
%\draw[white] (350:0.4cm) ++(4cm,0pt) circle (0.4pt) coordinate (g1);
%\filldraw[red] (180:1.5cm) ++(4cm,0pt) circle (1pt) coordinate (d);
%\filldraw[red] (240:1.5cm) ++(4cm,0pt) circle (1pt) coordinate (e);
%\filldraw[red] (300:1.5cm) ++(4cm,0pt) circle (1pt) coordinate (f);
%\filldraw[red] (0:1.5cm) ++(4cm,0pt) circle (1pt) coordinate (g);
%\draw[semithick,blue!30!white] (4,1.5) to (4,-1);
%\draw[semithick,blue!30!white] (4,1.5) to (a);
%\draw[semithick,blue!30!white] (4,1.5) ..controls (d1) and (e1).. (b);
%\draw[semithick,blue!30!white] (4,1.5) ..controls (f1) and (g1).. (c);
%\draw[semithick,blue!30!white] (4,1.5) to (4.9,0.3);
%\draw[semithick,blue!30!white] (4,1.5) to (3.1,0.3);
%\draw[semithick, red] (d) to[bend left] (e);
%\draw[semithick, red] (e) to[bend left] (f);
%\draw[semithick, red] (f) to[bend left] (g);
%\draw[semithick, red] (d) to (3.1,0.3);
%\draw[semithick, red] (g) to (4.9,0.3);
%\filldraw[red] (4.9,0.3) circle(0.02cm);
%\filldraw[red] (3.1,0.3) circle(0.02cm);
%\node at  (4.9,0.3) {\tiny$\times$};
%\node at  (3.1,0.3) {\tiny$\times$};
\end{tikzpicture}
    \caption{Dissections for the skew-gentle algebra and its Koszul dual associated to $\mathbb{D}_5$ on the left and $A_2$ on the right from Example\ref{ex:D_and_tildeD_are_skew-gentle}}
    \label{Fig:types_D5_and_tildeD6_dualGraphs}
\end{figure}
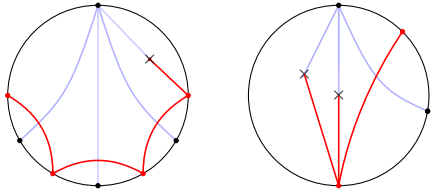

\begin{exmp}
Let $Q$ be the quiver\\
\begin{figure}[ht]
\centering
\begin{tikzcd}
 & \bullet \arrow[dr,"\alpha_2"] \arrow[out=35,in=125,loop,swap, "\varepsilon"] & \\
\bullet \arrow[ur, "\alpha_1"]    &    &  \bullet \arrow[ll, "\alpha_3"]
\end{tikzcd}
\end{figure}
\\
and $\mathcal{R}_1=\{\alpha_1\alpha_2,\alpha_2\alpha_3,\alpha_3\alpha_1,\varepsilon^2-\varepsilon\}$ and $\mathcal{R}_2=\{\alpha_1\alpha_2,\alpha_2\alpha_3, \varepsilon^2-\varepsilon\}$. The algebras $A_1=KQ/\langle\mathcal{R}_1\rangle$ and $A_2=KQ/\langle \mathcal{R}_2\rangle$ are skew-gentle. Orbifold dissections and their duals are depicted in Figure \ref{Fig:ribbon-and-dual-3cycle-with-loop}.
\begin{figure}[ht!]
    \centering
    \begin{tikzpicture}[scale=0.8]
%%%%First figure
\node at (2cm,-2.5cm) {Orbifold dissection and its dual graph of $A_1$};
\draw (0,0) circle (1.5cm);
\draw[semithick,blue!50!white] (90:1.5cm) to[out=180,in=180] (270:1.5cm);
\draw[semithick,blue!50!white] (90:1.5cm) to[out=0, in=0] (270:1.5cm);
\draw[semithick, blue!50!white] (90:1.5cm) to (0,0);
\draw[pattern=north west lines, thick] (270:0.7) circle(5pt);
\draw[pattern=north west lines, thick] (270:0.7) circle(5pt);
\filldraw (90:1.5cm) circle (1pt);
\filldraw (270:1.5cm) circle (1pt);
\node at (0,0) {\tiny$\times$};
%%%%% SECOND FIGURE
\draw (4,0) circle (1.5cm);
\filldraw (90:1.5cm) ++(4,0)circle (1pt) coordinate (a);
\filldraw (270:1.5cm) ++(4,0)circle (1pt) coordinate (b);
\draw[semithick,blue!20!white] (a) to[out=180,in=180] (b);
\draw[semithick,blue!20!white] (a) to[out=0, in=0] (b);
\draw[semithick, blue!20!white] (a) to (4,0);
\filldraw (a) circle (1pt);
\filldraw (b) circle (1pt);
\node at (4,0) {\small$\times$};
\filldraw[semithick, red] (180:1.5) ++(4,0) circle (1pt) coordinate (c);
\filldraw[semithick, red] (0:1.5) ++(4,0) circle (1pt) coordinate (d);
\filldraw[semithick, red] (270:0.9)++(4,0) circle (1pt) coordinate (e);
\draw[semithick, red] (e) to (c);
\draw[semithick, red] (e) to (d);
\draw[semithick, red] (e) to (4,0);
\filldraw[semithick, red] (4,0) circle (0.02);
%%%%%%Third figure
\node at (11cm,-2.5cm) {Orbifold dissection and it dual graph of $A_2$};
\draw (9,0) circle (1.5cm);
\draw[pattern=north west lines, thick] (9,0) circle(0.3cm);
\filldraw (45:1.5cm) ++(9,0)circle (1pt) coordinate (f);
\filldraw (0:0.3cm) ++(9,0)circle (1pt) coordinate (g);
\filldraw (0:1) ++(9,0) circle (0.01cm) coordinate (i);
\node at (i) {\small$\times$};
\draw[semithick,blue!40!white] (f) to (g);
\draw[semithick,blue!40!white] (g) to (i);
\draw[semithick, blue!40!white] (g) to[out=45, in=0] (9,0.75);
\draw[semithick, blue!40!white] (g) to[out=-45, in=0] (9,-0.75);
\draw[semithick, blue!40!white] (9,0.75) to[out=170, in=90] (8.3,0);
\draw[semithick, blue!40!white] (9,-0.75) to[out=-170, in=-90] (8.3,0);
\filldraw (f) circle (1pt);
\filldraw (g) circle (1pt);
%%%Fourth figure
\draw (13,0) circle (1.5cm);
\draw[pattern=north west lines, thick] (13,0) circle(0.3cm);
\filldraw (45:1.5cm) ++(13,0)circle (1pt) coordinate (h);
\filldraw (0:0.3cm) ++(13,0)circle (1pt) coordinate (j);
\filldraw (0:1) ++(13,0) circle (0.01cm) coordinate (k);
\node at (k) {\small$\times$};
\draw[semithick,blue!20!white] (h) to (j);
\draw[semithick,blue!20!white] (j) to (k);
\draw[semithick, blue!20!white] (j) to[out=45, in=0] (13,0.75);
\draw[semithick, blue!20!white] (j) to[out=-45, in=0] (13,-0.75);
\draw[semithick, blue!20!white] (13,0.75) to[out=170, in=90] (12.3,0);
\draw[semithick, blue!20!white] (13,-0.75) to[out=-170, in=-90] (12.3,0);
\draw[semithick, red] (13,-1.5) to[out=30, in=-60] (14.2,0.5);
\draw[semithick, red] (12.1,0.5) to[out=-120, in=150] (13,-1.5);
\draw[semithick, red] (14.2,0.5) to[out=120, in=60] (12.1,0.5);
\filldraw[red] (13,-1.5) circle (1pt);
\filldraw[red] (13,-0.3) circle(1pt);
\draw[semithick, red] (13,-1.5) to (13,-0.3);
\draw[semithick, red] (13, -1.5) to (k);
\filldraw[red] (k) circle (1pt);
\filldraw (h) circle (1pt);
\filldraw (j) circle (1pt);
\end{tikzpicture}
    \caption{Orbifold dissection and dual graph for $A_1$ and $A_2$}
    \label{Fig:ribbon-and-dual-3cycle-with-loop}
\end{figure}
\end{exmp}

\section{Graded curves in an orbifold dissection}\label{subsec:decorated curves}

In this section we define graded curves in an orbifold with marked points. 
We begin by recalling from \cite{Chas-Gadgil} the notion of homotopy in an orbifold based on what is called skein relations in that paper. 

Let $A$ be a skew-gentle algebra and  $O$ be the associated orbifold. We  recall the notion of $O$-free homotopy from~\cite{Chas-Gadgil}.

\begin{defn}\label{Def: orbi-homotopic}
Two  oriented closed curves $\gamma$ and $\gamma'$ in  $O$ 
are 
$O$-homotopic if they are related by a finite number of moves given by either a homotopy in the complement of the orbifold points or  are related by moves   taking place in a disk $D$ containing exactly one orbifold point $o_x$ as in Figure~\ref{fig:skein_relation1}. That is,  a segment of a curve with no self-intersection in $D$ and passing through $o_x$ is $O$-homotopic relative to its endpoints to a segment spiralling around $\omega$ in either direction exactly once as in Figure~\ref{fig:skein_relation1}. 
\end{defn}

\begin{figure}[ht!]
    \centering
   \begin{tikzpicture}[thick,scale=0.65, every node/.style={scale=0.65}]
       \draw (0,0) circle (1.5cm);
       \node at (0,0) {$\times$};
       \draw (-1.5,0) to[out=55, in=95] (0.75,0);
      \draw[<-] (-1.5, 0) to[out=-55, in=-95] (0.75,0);
      \node at (2,0) {$\sim$};
       \draw (4,0) circle (1.5cm);
       \node at (4,0) {$\times$};
       \draw[<-] (2.5,0) to[out=55, in=95] (4,0);
      \draw (2.5, 0) to[out=-55, in=-95] (4,0);
     \node at (6,0) {$\sim$};
       \draw (8,0) circle (1.5cm);
       \node at (8,0) {$\times$};
       \draw[<-] (6.5,0) to[out=55, in=95] (8.75,0);
      \draw (6.5, 0) to[out=-55, in=-95] (8.75,0);
       \end{tikzpicture}
    \caption{Moves in a disk containing exactly one orbifold point.}
    \label{fig:skein_relation1}
\end{figure}
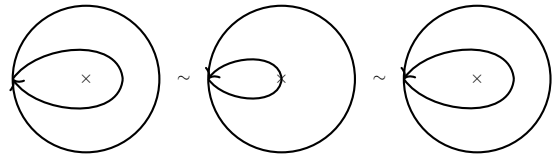

As a consequence of Definition \ref{Def: orbi-homotopic}, we have the following $O$-homotopic curves in $O$. 
 \begin{figure}[ht!]
     \centering
    \hspace*{0cm}    \begin{tikzpicture}[thick,scale=0.65, every node/.style={scale=0.65}]
       \draw (0,0) circle (1.5cm);
       \node at (0,0) {$\times$};
        \draw (-1.5,0) to[out=55, in=95] (0.75,0);
        \draw[<-] (-1.5,0) to[out=-55, in=-120] (0.35,-0.2);
        \draw(-0.35, -0.2) to[out=-65, in=-95] (0.75,0);
        \draw (-0.35,-0.2) .. controls (-0.4, -0.025) and (-0.175,0.15) .. (0,0.15)
               .. controls (0.175, 0.15) and (0.4,0.025) .. (0.35,-0.2);
        \node at (2,0) {$\sim$};
       \draw (4,0) circle (1.5cm);
       \node at (4,0) {$\times$};
       \draw[<-] (2.5,0) to[out=55, in=95] (3.5,0);
      \draw (2.5, 0) to[out=-55, in=-95] (3.5,0);
     \node at (6,0) {$\sim$};
       \draw (8,0) circle (1.5cm);
       \node at (8,0) {$\times$};
        \draw[<-](6.5,0) to[out=55, in=95] (8.75,0);
        \draw (6.5,0) to[out=-55, in=-120] (8.35,-0.2);
        \draw(7.65, -0.2) to[out=-65, in=-95] (8.75,0);
        \draw (7.65,-0.2) .. controls (7.6, -0.025) and (7.825,0.15) .. (8,0.15)
               .. controls (8.175, 0.15) and (8.4,0.025) .. (8.35,-0.2);
       \end{tikzpicture}
         \caption{$O$-homotopic curves in a disk containing one orbifold point.}
    \label{fig:skein_relation2}
\end{figure}
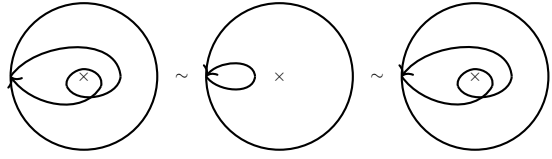

\begin{defn}\label{Def:arcs}
Let $O=(S,M,\orbi)$  be an orbifold with set $\orbi$ of orbifold points of order 2 with a finite set $M$ of marked points in the boundary component of $S$ or in the interior of $S$. Let $x, y$ be marked points.

\begin{enumerate}
    \item If $x$ and $y$ are marked points in the boundary of $S$, a  \emph{finite arc} (or simply an arc) $\gamma$ from $x$ to $y$ is an $O$-homotopy class, relative to endpoints, of non-contractible curves $\gamma$ from $x$ to $y$ in $O$.
\item A \emph{closed curve} $\gamma$  is a free $O$-homotopy class of non-contractible closed curves $\gamma$ not passing through any orbifold points.

\item
An \emph{infinite arc} is an $O$-homotopy class $\gamma$ associated to rays, that is a continuous maps $r:(0,1]\to O$ (or $r:[0,1) \to O$) or $l:(0,1) \to O$ respectively, that wrap around an unmarked boundary component in a clockwise way, asymptotically approaching this boundary.

Recall that two rays $r : (0, 1] \to  O$ and $r' : (0, 1] \to  O$ are 
$O$-homotopic if they wrap infinitely many times around the same unmarked boundary component $B$, their endpoints coinciding in the marked point $x$ in the boundary, and if for every closed 
neighbourhood $N$ of $B$ the induced maps $r,r' : \left[0,1\right] \to O \setminus N$ are $O$-homotopic 
relative to their endpoints. Similarly, we say two lines $l : (0,1)\to  O$ and $l' : (0, 1) \to O$ are equivalent if they wrap infinitely many times around the same unmarked boundary components $B$ and $B'$ on either end and if for every closed neighbourhood $N$ of $B$ and $N'$ of $B'$ the 
induced maps $l, l' : [0, 1] \to  O \setminus (N \cup N')$ are $O$-homotopic relative to their endpoints.
\end{enumerate}
\end{defn}

Our main result requires a notion of grading on  arcs and closed curves. This grading depends on the dual graph $\gra^*_A$ of a skew-gentle algebra $A$. Therefore before giving the definition, we will need some results on the geometry of the graph $\gra^*_A$. Moreover, since every unmarked boundary in a orbifold dissection $(S,M, \orbi, G)$ of $A$ is replaced by a vertex of $G^*$, we will view unmarked boundary components as marked points in the interior. In particular, it will be useful to think of infinite arcs wrapping around a boundary component  as infinite arcs wrapping around a marked points in the interior. We note that in our model we only consider infinite arcs as in Definition~\ref{Def:arcs}(3), that is  only those infinite arcs that wrap around an unmarked boundary component (ie a marked point in the interior of the surface) in a clockwise way.

\begin{lem}\label{lem:dual-dissection}
Let $A=KQ/I$ be a skew-gentle algebra with orbifold $O_A = (S,M,\orbi)$ with set of marked points $M$ and embedded  generalised ribbon graph $\gra_A$. Then the dual graph $\gra^*_A$ subdivides $O_A$ into polygons and degenerate polygons, where the edges of each such polygon are edges of $\gra^*_A$ and exactly one boundary segment containing exactly one marked point of $M$. 
\end{lem}

\begin{proof}
Let $\Lambda$ be the gentle algebra obtained from $A$ by  deleting all special loops in  $Q$ and let $(S, M_\Lambda, \gra_\Lambda)$ be   surface dissection associated to $\Lambda$.
By \cite[Lemma 2.6]{OPS18} and Remark \ref{Rem:lamination}, the dual graph $\gra^*_{\Lambda}$ of $\Lambda$ subdivides $S_{\Lambda}$ into polygons and degenerate polygons, where the edges of each such polygon are edges of $\gra^*_\Lambda$ and exactly one boundary segment containing exactly one marked point of $M_\Lambda$. 

It is enough to observe that after a local replacement of a special edge $\gamma$ in $\gra_{\Lambda}$, the unique edge $\gamma^*$ of $\gra^*_\Lambda$ crossing $\gamma$ corresponds to a special edge in $\gra^*$. By definition $\gamma^*$ is connected to an orbifold point $o_x$. 
Let $P$ and $P'$ be the polygons in of $\gra^*_\Lambda$ in $S$ sharing the common edge $\gamma^*$. As a consequence of the local replacement,  $P$ and $P'$ correspond to a single degenerate polygon in $S$ containing $o_x$, see Figure~\ref{fig:local-replacement-dual}. The result follows. 
\begin{figure}
 \centering
    \begin{tikzpicture}
\draw[dashed] (90:1cm) to[bend left](0:1cm);
\draw[red] (180:1cm) to[bend left] (-0.7,-0.7);
\draw[red] (-0.7,-0.7) to[bend left] (270:1);
\draw[blue!40!white] (80:0.96) to[bend right] node[below] {\textcolor{black}{\tiny$\gamma$}} (10:0.96);
\draw[red] (180:1) to[bend right] (145:1);
\draw[red] (270:1) to[bend left](305:1cm);
\draw[red] (45:0.9) to node[above] {\textcolor{black}{\tiny$\gamma^*$}} (-0.7, -0.7);
\filldraw[red] (45:0.9) circle (1pt);
\filldraw (80:0.96) node[above] {\tiny$w$} circle (1pt);
\filldraw (10:0.96) node[right] {\tiny$w'$} circle (1pt);
\filldraw[red] (-0.7,-0.7) circle (1pt);
\filldraw[red](180:1) circle (1pt);
\filldraw[red](270:1) circle (1pt);
\node at (-0.6,0) {\tiny$P$};
\node at (0,-0.4) {\tiny $P'$};
\draw[blue!40!white] (80:0.96) to[bend left] (95:0.75);
\draw[blue!40!white] (10:0.96) to[bend right] (355:0.75);
%%%%Second figure
\filldraw[red] (45:0.9)++(4,0) node[right] {\tiny$w=w'$} circle (1pt) coordinate (a);
%\filldraw (80:0.96)++(4,0) circle (1pt) coordinate (b);
%\filldraw (10:0.96)++(4,0) circle (1pt) coordinate (c);
\filldraw[red] (3.3,-0.7) circle (1pt) coordinate (d);
\filldraw[red](180:1)++(4,0) circle (1pt) coordinate (e);
\filldraw[red](270:1)++(4,0) circle (1pt) coordinate (f);
\filldraw[red] (145:1)++(4,0) circle (0.001) coordinate (g);
\filldraw[red] (305:1)++(4,0) circle (0.001) coordinate (h);
\filldraw[blue!40!white] (95:0.75)++(4,0) circle (0.001) coordinate (i);
\filldraw[blue!40!white] (355:0.75)++(4,0) circle (0.001) coordinate (j);
\draw[dashed] (4,1) to[bend left](5,0);
\draw[red] (e) to[bend left] (3.3,-0.7);
\draw[red] (3.3,-0.7) to[bend left] (f);
\draw[blue!40!white] (a) to node[below] {\textcolor{black}{\tiny$\gamma$}} (4,0);
\draw[red] (e) to[bend right] (g);
\draw[red] (f) to[bend left](h);
\draw[red] (3.3, -0.7) to node[above] {\textcolor{black}{\tiny$\gamma^*$}} (4,0);
\filldraw (45:0.9)++(4,0) circle (1pt);
\filldraw[red] (3.3,-0.7) circle (1pt);
\filldraw[red](e) circle (1pt);
\filldraw[red](f) circle (1pt);
\node at (3.4,0) {\tiny$P$};
\node at (4,-0.4) {\tiny $P'$};
\node at (4,0) {$\times$};
\draw[blue!40!white] (a) to[bend left] (i);
\draw[blue!40!white] (a) to[bend right] (j);
\end{tikzpicture}
\caption{}
\label{fig:local-replacement-dual}
\end{figure}
\end{proof}

By the Lemma~\ref{lem:dual-dissection}, $\gra^*_A$ induces a dissection of $O_A$. Furthermore, the generalised polygons of $\gra^*_A$ are in bijection with vertices of $G_A$.  

\begin{rem}\label{rem:minimalposition}
(1) We assume that any finite collection of curves is in  \emph{minimal position}, that is, the number of intersections of each pair of (not necessarily distinct) curves in this set is minimal. 

(2)  If $\gamma$ is an arc or closed curve in an orbifold dissction $\dissection$ we always assume that  $\gamma$ crosses every edge of $G$ transversely. 

(3) In Figure~\ref{fig:orbi-equivalent-curves} we give examples of $O$-homotopic curves.  In each case the first curve represents the chosen representative in its $O$-homotopic class which we will usually be working with. 

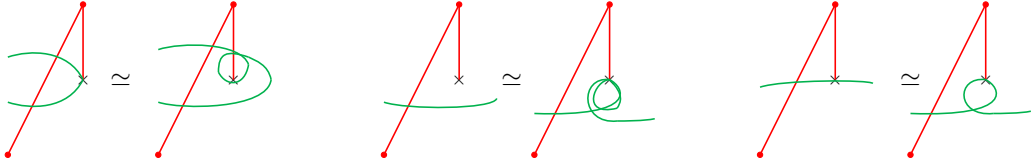
\begin{figure}[ht!]
    \centering
   \begin{tikzpicture}
%%%%% First figure
\filldraw[red] (0,0) circle (1pt);
\filldraw[red] (1,2) circle (1pt);
\node at (1,1) {\tiny$\times$};
\draw[semithick,red] (0,0) to (1,2);
\draw[semithick,red] (1,2) to (1,1);
\draw[semithick,blue!30!green] (0,1.3) ..controls (0.3,1.4) and (0.8,1.4)..(1,1);
\draw[semithick,blue!30!green] (0,.7) ..controls (0.3, 0.6) and (0.8, 0.6)..(1,1);
\node at (1.5,1) {$\simeq$};
%%%%% SECOND FIGURE
\filldraw[red] (2,0) circle (1pt);
\filldraw[red] (3,2) circle (1pt);
\node at (3,1) {\tiny$\times$};
\draw[semithick,red] (2,0) to (3,2);
\draw[semithick,red] (3,2) to (3,1);
\draw[semithick,blue!30!green] (2,1.4) ..controls (2.3,1.5) and (3.1,1.5)..(3.2,1.2);
\draw[semithick,blue!30!green] (2,.7) ..controls (2.3, 0.6) and (3.45, 0.6)..(3.5,1);
\draw[semithick,blue!30!green] (3.2,1.2) ..controls (3.15,0.9) and (2.9,0.9)..(2.8,1.1);
\draw[semithick,blue!30!green] (2.8,1.1) ..controls (2.8,1.35) and (2.9,1.4)..(3.2,1.3);
\draw[semithick,blue!30!green] (3.2,1.3) ..controls (3.3,1.25) and (3.46,1.2)..(3.5,1);
%%%%% Third FIGURE
\filldraw[red] (5,0) circle (1pt);
\filldraw[red] (6,2) circle (1pt);
\node at (6,1) {\tiny$\times$};
\draw[semithick,red] (5,0) to (6,2);
\draw[semithick,red] (6,2) to (6,1);
\draw[semithick,blue!30!green] (5,.7) ..controls (5.3, 0.6) and (6.45, 0.6)..(6.5,0.75);
\node at (6.7,1) {$\simeq$};
%%%%% fourt FIGURE
\filldraw[red] (7,0) circle (1pt);
\filldraw[red] (8,2) circle (1pt);
\node at (8,1) {\tiny$\times$};
\draw[semithick,red] (7,0) to (8,2);
\draw[semithick,red] (8,2) to (8,1);
\draw[semithick,blue!30!green] (7,.55) ..controls (8.4, 0.5) and (8.2, 0.9)..(8,1);
\draw[semithick,blue!30!green] (8,1) ..controls (7.8,1) and (7.6,0.5)..(8.1,0.62);
\draw[semithick,blue!30!green] (8.1,0.62) ..controls (8.15,0.7) and (8.2, 0.8)..(8,1);
\draw[semithick,blue!30!green] (8,1) ..controls (7.6,1.05) and (7.6,0.4)..(8.1,0.45);
\draw[semithick,blue!30!green] (8.1,0.45) ..controls (8.2,0.45) and (8.46,0.45)..(8.6,0.48);
%%%%%5to FIGURE
\filldraw[red] (10,0) circle (1pt);
\filldraw[red] (11,2) circle (1pt);
\node at (11,1) {\tiny$\times$};
\draw[semithick,red] (10,0) to (11,2);
\draw[semithick,red] (11,2) to (11,1);
\draw[semithick,blue!30!green] (10,0.9) ..controls (10.3, 1) and (11.45, 1)..(11.5,0.95);
\node at (12,1) {$\simeq$};
%%%%%5to FIGURE
\filldraw[red] (12,0) circle (1pt);
\filldraw[red] (13,2) circle (1pt);
\node at (13,1) {\tiny$\times$};
\draw[semithick,red] (12,0) to (13,2);
\draw[semithick,red] (13,2) to (13,1);
\draw[semithick,blue!30!green] (12,.55) ..controls (13.4, 0.5) and (13.2, 0.9)..(13,1);
\draw[semithick,blue!30!green] (13,1) ..controls (12.6,1.05) and (12.6,0.5)..(13.1,0.55);
\draw[semithick,blue!30!green] (13.1,0.55) ..controls (13.2,0.55) and (13.46,0.55)..(13.6,0.58);
\end{tikzpicture}
    \caption{$O$-homotopic curves}
    \label{fig:orbi-equivalent-curves}
\end{figure}

\end{rem}

\begin{defn}
Given a skew-gentle algebra $A$ and the associated graph $\gra_A$ with vertex set $M$, with dual graph $\gra^*_A$ and associated orbifold $O = (S,M, \orbi)$, we call the tuple $(S, M, \orbi, \gra^*_A)$ the \emph{orbifold dissection associated} to $A$. 
\end{defn}

We note that if  $(S, M,\orbi, \gra^*)$ is an orbifold dissection and if $\gamma$ is a possibly infinite arc  or a closed curve in $O$ then $\gamma$ is completely determined by the possibly infinite sequence of edges of $\gra^*$ which it crosses. If $\gamma$ is a closed curve then this sequence is determined up to cyclic permutation. 

\begin{defn}
Let $\gamma$ be an arc or a closed curve in an  orbifold dissection $(S,M,\orbi, \gra)$ and let $\left(x_i\right)$ be the ordered \emph{multiset of edges} of the dual graph $\gra^*$ of $\gra$  crossed successively by $\gamma$.

Let $\left[x_{i}, x_{i+1}\right]$ be the oriented segment going from $x_i$ to $x_{i+1}$. Then  both $x_i$ and $x_{i+1}$ are edges of the same  (degenerate) polygon $P$ which contains exactly one marked point $m \in M$.  A \emph{grading} on $\gamma$ is a function $f:\left(x_i\right) \to \mathbb Z$ satisfying the following conditions.

$$f(x_{i+1})=
\begin{cases}
f(x_i)+1 & \textrm{if $m$ is to the left of $\left[x_{i}, x_{i+1}\right]$ in $P$;} \\
f(x_i)-1 & \textrm{if $m$ is to the right of $\left[x_{i}, x_{i+1}\right]$ in $P$}
\end{cases}$$

If $\gamma$ is an (infinite) arc, we say that $(\gamma, f)$ is an (infinite) \emph{graded arc} on $S_{A}$.
If $\gamma$ is a closed curve successively crossing the edges $x_1, \dots, x_n$ of $G^*$,
we say that $(\gamma, f)$ is a \emph{graded closed curve} if the grading  $f$ is such that $f(x_1) = f(x_n)$. 
\end{defn}

\section{Indecomposable objects in the derived category of a skew-gentle algebra}\label{sec:indecomposables-objects}

In this section, given a skew-gentle algebra $A$, we show that the geometric model constructed in Section \ref{sec:geometricmodel} is a model for the bounded derived category $D^b(A)$. More precisely, using the equivalence $D^b(A)$ and $K^{-,b}(proj-A)$, we  establish a one to one correspondence between the homotopy strings and bands encoding the indecomposable objects in $K^{-,b}(proj-A)$ and graded arcs and curves in the orbifold associated to $A$.

\subsection{Homotopy strings and bands}\label{subsec:homotopystrings}

We begin by  briefly recalling the definition of homotopy strings and bands from \cite{BMM03}. Throughout this section let $A$ be a skew-gentle algebra and let  $\Lambda=KQ/I$ be the associated gentle algebra obtained from $A$ by deleting all special loops.

The definition of homotopy strings and bands for skew-gentle algebras is based on another gentle algebra $A_+$ underlying $A$ in the following way. \yad{Let $B$ a minimal set of relations of $\Lambda=KQ/I$} and  set $J  = \langle B\setminus \{\alpha\beta\mid \alpha\beta\in I, t(\alpha)\in Sp\} \rangle$. Then by \cite{BMM03} $A_+= KQ/ J$ is a gentle algebra. 

For every arrow $\alpha \in Q_1$, we define a formal inverse arrow $\overline{\alpha}$ where $s(\overline{\alpha}) = t(\alpha)$ and $t(\overline{\alpha}) = s(\alpha)$. For each path $w=\alpha_1\dots\alpha_k$ we define $\overline{(\alpha_1\dots \alpha_k)}=\overline{\alpha_k}\dots \overline{\alpha_1}$, $s(\overline{w})=e(w)$ and $t(\overline{w})=s(w)$.
A \emph{walk} $w$ is sequence $w_1 \dots w_n$ where $w_i$ is either an arrow or an inverse arrow such that $s(w_{k+1})=e(w_k)$. 

A \emph{string} is a walk $w=w_1\dots w_n$ such that $w_{k+1}\neq  \overline{w_k}$ for $1\leq k <n$ and such that no substring $w'$ of $w$ or its inverse $\overline{w'}$ is in $J$. For every $i\in Q_0$,  we denote by $e_i$ the string corresponding to the \emph{trivial walk} at $i$. 

A string $w=w_1 \dots w_n$ is a \emph{direct } (resp. \emph{inverse}) \emph{homotopy letter}  if  $w_k$ is a direct (resp. inverse) arrow, for all $1\leq k \leq n$.
A \emph{homotopy walk} $\sigma$ is a sequence $\sigma_1\dots \sigma_r$, where $\sigma_k$ is a direct or inverse homotopy letter such that $s(\sigma_{k+1}) = t(\sigma_k)$. We say that $\sigma=\sigma_1\dots \sigma_r$ is a \emph{direct}  (resp. \emph{inverse}) homotopy walk if  $\sigma_k$ is  direct (resp. inverse), for all $1\leq k \leq r$.
A walk $\sigma$ (resp. a homotopy walk) is \emph{closed} if $s(\sigma)=t(\sigma)$. Given a (homotopy) walk $\sigma=\sigma_1 \dots \sigma_r$ we denote by $\sigma[m]=\sigma_{m+1} \dots \sigma_r\sigma_1 \dots \sigma_j$  its \emph{rotations} where $m=1, \dots, r-1.$

\begin{defn}
Let $A$ be a skew-gentle algebra with $\Lambda=KQ/I$ be the associated gentle algebra obtained from $A$ be deleting all special loops, and $A_+=KQ/J$ the associated gentle algebra as defined above. 
\begin{enumerate}
\item A \emph{homotopy string} is a homotopy walk $\sigma=\sigma_1\dots \sigma_r$ such that
\begin{itemize}
    \item if both $\sigma_i, \sigma_{i+1}$ are direct (resp. inverse) homotopy letters such that $\sigma_i$  (resp. $\overline{\sigma_{i+1}}$) is not ending at a special vertex, then $\sigma_i \sigma_{i+1}\in  J $ (resp. $\overline{\sigma_i \sigma_{i+1}} \in J$),
 \item if $\sigma_i, \overline{\sigma_{i+1}}$   (resp. $\overline{\sigma_{i}}, \sigma_{i+1}$) are direct 
  homotopy letters and $\sigma_i$ (resp. $\overline{\sigma_{i}}$) is not ending at a special vertex, 
  then $\sigma_i \sigma_{i+1}$ is a string.
\end{itemize}

A nontrivial homotopy string $\sigma$ is \emph{symmetric} if $\sigma=\overline{\sigma}$ and \emph{asymmetric} otherwise.

\item A \emph{homotopy band} is a closed homotopy string $\sigma=\sigma_1 \dots \sigma_{n}$ with an equal number of direct and inverse homotopy letters  such that $\sigma$  is not  a proper power of some homotopy string $\sigma '$ and such that every power of $\sigma$ is a homotopy string.

 A nontrivial homotopy band $\sigma$ is \emph{symmetric} if $\sigma= \overline{\sigma}[m]$ for some $m$ and \emph{asymmetric} otherwise.

\end{enumerate}
\end{defn}

\begin{rem}
By definition, any symmetric band $\sigma$ is a word such that $$\sigma=\sigma=\sigma_1 \dots \sigma_k=a_1 \dots a_r\overline{a_r}\dots \overline{a_1}b_1 \dots b_s \overline{b_s} \dots \overline{b_1},$$ where $t(a_r)$ and $t(b_s)$ are special vertex, having the following picture:

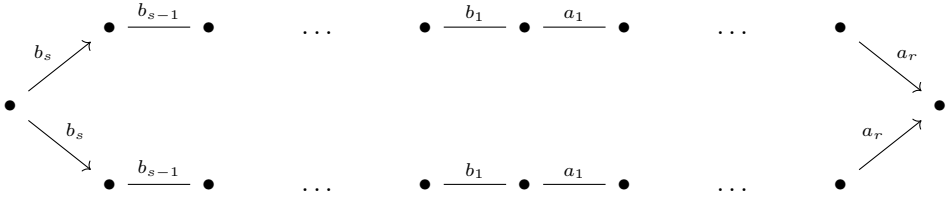
\begin{tikzcd}
     \     & \bullet \arrow[r, no head, "b_{s-1}"] & \bullet & \dots & \bullet \arrow[r,no head, "b_1"] &\bullet \arrow[r,no head, "a_1"] & \bullet & \dots & \bullet \arrow[rd, "a_r"] &  \\
     \bullet \arrow[ru, "b_s"] \arrow[dr, "b_s"] & \ & \ & \ &\ & \ & \ & \ & \ &\bullet \\
     \ & \bullet \arrow[r,no head, "b_{s-1}"] & \bullet & \dots & \bullet  \arrow[r,no head, "b_1"] &\bullet \arrow[r,no head, "a_1"] & \bullet & \dots & \bullet \arrow[ru, "a_r"] & 
\end{tikzcd}

\end{rem}

\begin{defn}
A \emph{right (resp. left) infinite homotopy string} is a sequence $\sigma = \sigma_0 \sigma_1 \sigma_2 \ldots $ (resp. $\sigma= \dots  \sigma_{-2}\sigma_{-1} \sigma_0 $ ) such that for some $k$, all $\sigma_i$  (resp. $\sigma_{-i}$) are direct (resp. inverse) arrows, for $i >k$ and such that every finite subword of $\sigma$ is a homotopy string.  

An \emph{infinite homotopy string} is a sequence $\sigma= \dots \sigma_{-2}\sigma_{-1}\sigma_0\sigma_1 \sigma_2 \dots$ such that $\dots \sigma_{-2}\sigma_{-1}$ is a left infinite homotopy string and $\sigma_0\sigma_1\sigma_{2}$ is a right infinite homotopy string.
\end{defn}

\subsection{String and Band Complexes}

In order to define the complexes induced by homotopy strings and bands as introduced in \cite{BMM03}, we need to introduce a  grading on homotopy strings and bands. Our definition closely follows  \cite{OPS18}.

\begin{defn}
Let $\sigma=\sigma_1 \dots \sigma_r$ be a finite  homotopy string.

A \emph{grading} on $\sigma$ is a sequence of integers $\mu=(\mu_0, \dots, \mu_r)$ such that $$\mu_{i+1}=\begin{cases} \mu_i+1 & \textrm{if $\sigma_{i+1}$ is a direct homotopy letter;}\\ \mu_{i}-1 & \textrm{otherwise,}\end{cases}$$
for each $i\in \{1,\dots, r-1\}$. The pair $(\sigma, \mu)$ is a \emph{graded homotopy string.}

Moreover, if $\sigma$ is a homotopy band, the pair $(\sigma, \mu)$ is a \emph{graded homotopy band} if $(\sigma, \mu)$ is a graded homotopy string and $i$ is considered modulo $r$.
\end{defn}

In a similar way, we define a grading on (left, right) infinite homotopy strings.

Let $A$ be a skew-gentle algebra with admissble presentation $A^{sg} = KQ^{sg} / I^{sg}$. For a vertex $i \in Q_0^{sg}$,  we write $P_i$ for the projective indecomposable $A^{sg}$-module at vertex $i$.  

Following \cite{BMM03}, to each graded homotopy string or band $(\sigma, \mu)$ we associate a  complex of projective $A$-modules 
$P^\bullet_{(\sigma, \mu)}$ which is not necessarily indecomposable, but is a sum of at most two indecomposable  complexes. For this,  we freely view $A$-modules as modules over $A^{sg}$ and we define the following projective $A$-module for every $i \in Q_0$.

$$P(i)=\begin{cases}
P_{i^+}\oplus P_{i^-} & \textrm{if $i$ is a special vertex;}\\
P_i & \textrm{otherwise.}
\end{cases}$$

The definition of a complex associated to a symmetric homotopy band $\sigma$ relies on a set of matrices $\mathcal{M}$ with coefficients in $K$ where $M \in \mathcal{M}$ is such that $M=\begin{pmatrix} A & C \\ B & D\end{pmatrix}\in \textrm{Mat}(l+l', m+m')$ for some strictly positive integers $l, l',m,m'$.  We refer the reader to  \cite{BMM03}[Section 3.1] for the precise definition of $\mathcal{M}$.

\begin{defn} 
\begin{itemize}
    \item[(1)] Let $(\sigma,\mu)$ be a graded homotopy string with $\sigma=\sigma_1 \dots \sigma_r$. Then let  
    $$P^\bullet_{(\sigma, \mu)}= \dots \longrightarrow P^{-1} \longrightarrow P^0 \longrightarrow P^1 \longrightarrow \dots$$
   be the complex such that   $P^j=\bigoplus_{\substack{0\leq i \leq r \\ \mu_i=j}} P(i)$,  for all $j \in \mathbb Z$ and where the differentials are induced by the homotopy letters.
         
  If $\sigma$ is an asymmetric homotopy string,   we say that  a complex $P$ is \emph{an asymmetric string complex} if it is isomorphic to $P_{(\sigma, \mu)}$ in  $K^{b,-}(proj - A)$.
 
 If $\sigma$ is a symmetric homotopy string then  $P_{(\sigma, \mu)}$ decomposes into the direct sum of two indecomposables  complexes of projective $A$-modules. We will refer to the indecomposable summands as  $P_{(\sigma, \mu,0)}$ and $P_{(\sigma, \mu,1)}$. We call  a complex $P$ isomorphic to either $P_{(\sigma, \mu, 0)}^{\bullet}$ or $P_{(\sigma, \mu, 1)}^{\bullet}$, a \emph{dimidiate string complex.}

    \item[(2)] Let $(\sigma, \mu)$ be a graded homotopy band where  $\sigma=\sigma_1 \dots \sigma_r$ is an  asymmetric homotopy band. Let $\operatorname{ind}K\left[x\right]$  be the set of non trivial powers of irreducible polynomials over $K$ with leading coefficient equal to 1 and different from  $x$ and $x-1$. Then, for each $p(x)\in \operatorname{ind}k\left[x\right]$, let
     $$P^\bullet_{(\sigma, \mu), p(x)}= \dots \longrightarrow P^{-1} \longrightarrow P^0 \longrightarrow P^1 \longrightarrow \dots$$
    be the complex such that $P^j=\bigoplus_{\substack{0\leq i \leq r \\ \mu_i=j}} P(i) \otimes_K K^{\operatorname{deg} p(x)}$ in degree $j$.  
We call a complex  $P$  \emph{an asymmetric band complex} if it is isomorphic in $K^{b,-}(proj - A)$ to $P_{(\sigma, \mu), p(x)}$ where $\sigma$ is an asymmetric band.

\item[(3)] Let $(\sigma, \mu)$ be a graded homotopy band with $\sigma=a_1 \dots a_r\overline{a_r}\dots \overline{a_1}b_1 \dots b_s \overline{b_s} \dots \overline{b_1}$   symmetric and such that $e(a_r)$ and $e(b_s)$ are special vertices. Let $M=\begin{pmatrix} A & C \\ B & D\end{pmatrix}\in \mathcal M$. Then the complex 

 $$P^\bullet_{(\sigma, \mu)}= \dots \longrightarrow P^{-1} \longrightarrow P^0 \longrightarrow P^1 \longrightarrow \dots$$
is given by

 $$P^j=\bigoplus_{\substack{0\leq i \leq r \\ \mu_i=j}} Q(i)\oplus \bigoplus_{\substack{2r+1\leq i \leq 2r+s \\ \mu_i=j}} Q(i)$$
    
   for all $j\in \mathbb Z$, where $Q(i)$ is a projective $A$-module which depends on the size of the matrices $A,B,C,D$ in $M$  as follows:
    
    $$Q(i)=\begin{cases} P(i)\otimes_K K^{l} \oplus P(i)\otimes_K  K^{l'} & \textrm{if $i=\mu_0, \dots, \mu_{r-1}$}\\
P(i)\otimes_K K^{m} \oplus P(i)\otimes_K K^{m'} & \textrm{if $i=\mu_{2r+1}, \dots, \mu_{2r+s-1}$}\\
P_{i^+}\otimes_K K^{l} \oplus P_{i^-}\otimes_K K^{l'} & \textrm{if $i=\mu_{r}$}\\
P_{i^+}\otimes_K K^{m} \oplus P_{i^-}\otimes_K K^{m'} & \textrm{if $i=2r+s$}
\end{cases}$$

We say that a  complex $P$ is \emph{a dimidiate band complex} if it is isomorphic in  $K^{b,-}(proj - A)$ to $P_{(\sigma, \mu), M}$ where $\sigma$ is a symmetric band.
\end{itemize}
\end{defn}

We will not give the definitions of the differentials of the above complexes, since we do not need those in the geometric description of the indecomposable objects of the bounded derived category of a skew-gentle algebra. 

\begin{rem} In \cite{SV}  we give a correspondence of intersections of graded curves and homomorphisms in the bounded derived category of a skew-gentle algebra. \end{rem}

For a grading $\mu= (\mu_1, \ldots, \mu_r)$ on a homotopy string or band, define a grading shift $[m]$ as  $\mu[m] = (\mu_1 +m, \ldots, \mu_r +m)$ for $m \in \mathbb{Z}$. 
Observe that the the complex induced by $(\sigma, \mu[m])$ is $P^\bullet_{(\sigma, \mu[m])}=P^\bullet_{(\sigma, \mu)}[m]$.

\subsection{Main result on indecomposable objects of the derived category of  a skew-gentle algebra}

Before stating one of the main theorems of this paper, we recall that we identify unmarked boundary components and punctures in a surface $S$. We also recall that given an orbifold dissection $O=\dissection$ we define graded arcs and graded closed curves up to $O$-homotopy. 

\begin{thm}\label{thm:indecomposable objects}
Let $A$ be a skew-gentle algebra  with orbifold dissection $O=(S, M, \orbi, \gra^*)$. 
Then the homotopy strings and bands parametrizing the indecomposable objects in $D^b(A)$ are in bijection with  graded arcs and curves in $O$. More precisely, 

\begin{enumerate}
    \item the set of homotopy strings are in bijection with graded arcs $(\gamma,  f)$, where $\gamma$ is a finite arc in $O$ or an infinite arc whose infinite rays wrap around  unmarked boundary components  in the anti-clockwise orientation;
    \item the set of homotopy bands are in bijection with  graded primitive closed curves $(\gamma,  f)$ in $O$.\end{enumerate}
\end{thm}

The proof is very similar to that of the corresponding result for gentle algebras in \cite{OPS18} with suitable adjustments to be made for special loops in the algebra on the one side and polygons containing orbifold points in the surface on the other. For the convenience of the reader, we give the whole proof in detail. 

The following definition gives the construction of a homotopy  word induced by the interesection of an arc or curve in the surface with the edges of the dual graph.

\begin{defn}\label{def:homotopy letter}
Let $A$ be a skew-gentle algebra with quiver $Q_A=Q$ and  $\gamma$ be an  arc or a closed curve in its orbifold dissection $(S,M, \orbi,  \gra^*)$. We set  $\left(x_i\right)$ to be the ordered \emph{multiset of edges} of  $\gra^*$ crossed by $\gamma$ following its trajectory. Furthermore, if  $\gamma$ crosses $\gra^*$ at least twice. Let $\left[x_{i}, x_{i+1}\right]$ be the oriented segment going from $x_i$ to $x_{i+1}$, recall that both edges are in a (degenerate) polygon $P$ which contains exactly one marked point $m \in M$.

We define the \emph{homotopy letter} $\sigma(x_i)$ associated to the oriented segment $\left[x_i, x_{i+1}\right]$ as follows:

\begin{enumerate}
    \item If the marked point $m$ is on the left of $\left[x_i, x_{i+1}\right]$, then let $w_1, \dots, w_r$ be the edges between $x_i=w_1$ and $x_{i+1}=w_r$ in the clockwise order. By Theorem \ref{Thm:dual graph}, these correspond to vertices of the quiver $Q$ of $A$ which are joined by arrows $\alpha_1, \dots, \alpha_{r-1}$. Then define $\sigma(x_i):=(\alpha_1 \dots \alpha_{r-1})$.
    \item If the marked point $m$ is on the right of $\left[x_i, x_{i+1}\right]$, then let $w_1, \dots, w_r$ be the edges between $x_{i+1}=w_1$ and $x_i=w_r$ in the clockwise order. By Theorem \ref{Thm:dual graph}, these correspond to vertices of the quiver $Q$ of $A$ which are joined by arrows $\alpha_1, \dots, \alpha_{r-1}$. Then define $\sigma(x_i):=\overline{(\alpha_1 \dots \alpha_{r-1})}$.
\end{enumerate}
\end{defn}

If a graded arc $(\gamma, f)$ crosses $\gra^*$ exactly once, namely at the edge $x$, then we set $\sigma(\gamma)$ to be the trivial string $e_x$, and $\mu(f)=(f(x))$ the grading on $\sigma(\gamma)$. 

Observe that by Theorem \ref{Thm:dual graph}, $\sigma(x_i)$ is indeed a homotopy letter.

\begin{lem}\label{lem:from_arcs_to_strings}
Let $(S, M, \orbi, \gra^*)$ be the orbifold dissection of a skew-gentle algebra $A$ given by the dual $\gra^*$ of the generalised ribbon graph of $A$ and let $(\gamma, f)$ be an arc or a graded primitive closed curve on $(S,M,\orbi)$ with  $\left(x_i \right)$  the ordered multiset of edges of the dual graph $\gra^*$ crossed by $\gamma$ following its trajectory.
\begin{enumerate}

    \item If $(\gamma, f)$ is a finite graded arc which crosses  the edges of $\gra^*$ exactly $r$ times and at least twice, then 
    $\sigma(\gamma)= \prod_{i=1}^{r-1}\sigma(x_i)$
    is a homotopy string and $\mu(f) = (f(x_1),\dots,f(x_r))$ is a grading on $\sigma(\gamma)$.
    
    \item  If $(\gamma, f)$ is an infinite graded arc, then $\sigma(\gamma)=\prod\sigma(x_i)$ is a infinite homotopy string and $\mu(f) = (f(x_i))$ is a grading on $\sigma(\gamma)$.

    \item If $(\gamma, f)$ is a primitive graded closed curve and if   $x_1, x_2, \dots, x_r$  are the distinct edges of $\gra^*$ crossed by $\gamma$ (in that order), then $\sigma(\gamma)=\prod_{i=0}^r\sigma(x_i)$  is a homotopy band and $\mu(f) = (f(x_1),\dots,f(x_{r-1}))$ is a grading on $\sigma(\gamma)$.
\end{enumerate}
\end{lem}

\begin{proof}
 By construction, $\sigma(\gamma_i)$ is a homotopy letter, for all $i$. As before, let $\Lambda= KQ/I$ be the associated gentle algebra obtained from A be deleting all special loops, $B$ be a minimal set of relations of $\Lambda$ and $A+ = KQ/J$, where $J = \langle B \setminus \{\alpha \beta \mid t(\alpha) \in Sp \}\rangle$. By Theorem \ref{Thm:dual graph}, either the composition or  of the last arrow of $\sigma(\gamma_i)$ and the first arrow  of $\sigma(\gamma_{i+1})$  (or the composition of their inverses) are in $J$  or the end of $\sigma(\gamma_i)$ is a special vertex. Then  $(2)$ directly follows from the definition of homotopy strings in Section~\ref{subsec:homotopystrings}.

Now let $(\gamma, f)$ be an infinite graded arc. By Lemma~\ref{lem:dual-dissection}, $\gamma$ wraps around a puncture $p$ with at least one incident edge and by the above, every finite homotopy sub-walk of $\sigma(\gamma)$ is a homotopy string.
Denote by $x_k, x_{k+1}, \dots, x_r$ the set of edges of $\gra^*$ that $\gamma$ crosses when $\gamma$ does a complete turn around $p$. By Theorem \ref{Thm:dual graph}, their associated homotopy letters $\sigma(x_k), \dots, \sigma(x_r)$ are homotopy letters of length one, and $\sigma(x_k) \dots \sigma(x_r)$ or its inverse  is an oriented cycle.
Furthermore, we have that $\sigma(x_r) \sigma(x_{r+1}) \in J$ with $\sigma(x_{r+1}) = \sigma(x_{k})$ and more generally, $\sigma(x_{r+i}) = \sigma(x_{k+i-1})$ for all $i \geq 1$. Thus $\sigma(\gamma)$ is eventually periodic and an infinite homotopy string.

In cases (1) and (2),  directly follows from the definition of the grading $f$ on $\gamma$  that $\mu(f)$ is a grading on the associated homotopy string.

To prove $(3)$, assume that $(\gamma, f)$ is a graded primitive closed curve. Let $\sigma(\gamma)=\prod_{i=0}^r\sigma(\gamma_i)$. It is enough to observe that $s(\sigma(\gamma_1))=t(\sigma_{\gamma_r})$ to ensure that any rotation of $\sigma(\gamma)$ is a homotopy string. By definition, the existence of the  grading $f$ implies that there is the same number of inverse and direct homotopy letters in $\sigma(\gamma)$, which implies that $\sigma(\gamma)$ is a homotopy band.
\end{proof}

\begin{lem}\label{lem:from_strings_to_arcs}
Let $O=(S,M, \orbi, \gra^*)$ be the orbifold dissection of a skew-gentle algebra $A$ given by the dual graph generalised ribbon graph $\gra^*$ of $A$.
\begin{enumerate}
 \item Let  $(\sigma, \mu)$ be a finite (resp. infinite) graded homotopy string. Then there exists a unique finite (resp. infinite) graded arc $(\gamma, f)$ on $O$ such that $(\sigma(\gamma), \mu(f))=(\sigma, \mu)$.

\item Let $(\sigma,\mu)$ be a graded homotopy band. Then there exists a unique graded closed curve $(\gamma, f)$ on $O$ (up to $O$-homotopy) such that $(\sigma(\gamma), \mu(f))=(\sigma, \mu)$.
\end{enumerate}
\end{lem}

\begin{proof}

Let $\Lambda= KQ/I$ be the associated gentle algebra obtained from A be deleting all special loops, $B$ be a minimal set of relations of $\Lambda$ and $A+ = KQ/J$, where $J = \langle B \setminus \{\alpha \beta \mid t(\alpha) \in Sp \}\rangle$.

Let $(\sigma, \mu)$ be a finite graded homotopy string. By definition $\sigma$ is a sequence $\sigma_1 \dots \sigma_n$ of homotopy letters $\sigma_i$ where $s(\sigma_{i+1})=t(\sigma_i)$ and such that either $\sigma_i\sigma_{i+1}\in J$  or $\overline{\sigma_i\sigma_{i+1}}\in J$ or $s(\sigma_{i+1})=t(\sigma_i)$ is a special vertex.

It is enough to prove that for each homotopy letter $\sigma_i$ there is a (unique) oriented segment $\gamma_i$  such that the topological concatenation of those segments induces a graded arc $\gamma$ on $O$ and that two distinct homotopy strings give rise to two graded arcs which are not $O$-homotopic.

Denote by $x_i$ and $x_{i+1}$ the edges of $\gra^*$ corresponding to  $s(\sigma_i)$ and $t(\sigma_i)$ respectively. By construction of  $\gra^*$, there is exactly one (degenerate) polygon $P_i$ in $O$ such that  $x_i$ and $x_{i+1}$ are edges of $P_i$.

If $P_i$ is not degenerate or $x_i$ and $x_{i+1}$ are not special edges, then up to homotopy there is a unique oriented segment $\gamma_i$ in the interior of $P_i$ starting at  the mid-point of $x_i$ and ending at the mid-point of $x_{i+1}$ such that $\gamma_i$ has no self-crossing and does not cross any other edge of $\gra^*$.

If $x_i$ is not a special edge and $x_{i+1}$ is a special edge (or $x_i$ is a special edge and $x_{i+1}$ is not a special edge), then up to homotopy there is a unique oriented segment $\gamma_i$ in the interior of $P_i$ starting at the mid-point of (resp. the orbifold point of the edge) $x_i$ and ending at the orbifold point of the edge (resp. the mid-point of) $x_{i+1}$ such that $\gamma_i$ has no self-crossing and does not cross any other edge of $\gra^*$, see figure \ref{fig:segments}.

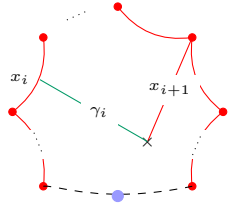
\begin{figure}[ht!]
    \centering
        \begin{tikzpicture}[scale=1.4]
        \filldraw[red](0:1cm) circle (1pt) coordinate (a);
		\filldraw[red](45:1cm) circle (1pt) coordinate (b);
		\filldraw[red](90:1cm) circle (1pt) coordinate (c);
		\filldraw[red](135:1cm) circle (1pt) coordinate (d);
		\filldraw[red](180:1cm) circle (1pt) coordinate (e);
		\filldraw[red](225:1cm) circle (1pt) coordinate (f);
		\filldraw[red](315:1cm) circle (1pt) coordinate (h);	
\draw[red] (a) to[bend left] (b);
\draw[red] (b) to[bend left]  (c);
\draw[red] (d) to[bend left] node[left] {\textcolor{black}{\tiny$x_i$}} (e);
\draw[red] (e) to[bend left] node[midway, ,fill=white] {\textcolor{white}{a}} (f);
\draw[red] (a) to[bend right] node[midway, ,fill=white] {\textcolor{white}{a}}  (h);
\draw[red] (b) to node[midway, ,fill=white] {\textcolor{black}{\tiny$x_{i+1}$}} (315:0.4cm);
\node at (315:0.4cm) {\textcolor{black}{\tiny$\times$}} ;
\draw[black, dashed] (f).. controls (230:0.9cm) and (260:0.9cm) ..(h);
\draw[blue!40!green] node[left, fill=white] {\textcolor{black}{\tiny$\gamma_i$}} (157.5:0.8cm) to (315:0.4cm);
\draw[dotted] (108:1cm) to (120:1cm);
\draw[dotted] (194:0.82cm) to (210:0.82cm);
\draw[dotted] (330:0.82cm) to (344:0.82cm);
\filldraw[blue!40!white](270:0.8cm) circle (1.5pt) coordinate (g);
\end{tikzpicture}
    \caption{The arc $\gamma$ from $x_i$ to $x_{i+1}$}
    \label{fig:segments}
\end{figure}

It is clear that the ending point of the oriented segment $\gamma_i$ is the starting point of $\gamma_{i+1}$ and that the concatenation $\gamma_{1}*\dots *\gamma_{n}$ is an oriented segment from $x_1=s(\sigma_1\dots \sigma_n)$ to $x_{n+1}=t(\sigma_1 \dots \sigma_n)$.

Now, let $P_0$ (resp. $P_{n+1}$)  be the (degenerate) polygon which shares $x_1$ (resp. $x_{n+1}$) with $P_1$ (resp. $P_n$) and let $m_0$ (resp. $m_{n+1}$) be the marked point in the boundary segment of $P_0$ (resp. $P_{x_{n+1}}$). Observe that if $x_1$ (resp. $x_n$) is a special edge, then $P_0$ and $P_1$ (resp. $P_n$ and $P_{n+1}$) coincide.

Suppose that $x_1$ (resp. $x_n$) is not a special edge. Then there exists a unique oriented segment from the marked point $m_0$ (resp. the point $x_{n+1}$) to middle point of $x_1$ (resp. marked point $m_{n+1}$), which lies in $P_0$ (resp. $P_{n+1}$), without self intersection and it is not crossing other edge of $G^*$.

If $x_1$ (resp. $x_n$) is a special edge, then there is a unique oriented segment $\gamma_0$ (resp. $\gamma_{n+1}$) from $m_0$ (resp. the middle point of $x_{n+1}$) to the orbifold point of $x_1$ (resp. $m_{n+1}$) such that $\gamma_0*\gamma_1$ (resp. $\gamma_{n}*\gamma_{n+1}$)lies in $P_1=P_0$ (resp. $P_{n}=P_{n+1}$) and is not crossing any other edge of $\gra^*$, see Figure \ref{fig:special_segment}.

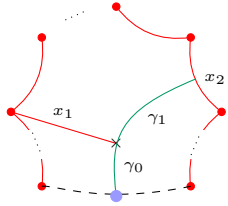
\begin{figure}[ht!]
    \centering
        \begin{tikzpicture}[scale=1.4]
        \filldraw[red](0:1cm) circle (1pt) coordinate (a);
		\filldraw[red](45:1cm) circle (1pt) coordinate (b);
		\filldraw[red](90:1cm) circle (1pt) coordinate (c);
		\filldraw[red](135:1cm) circle (1pt) coordinate (d);
		\filldraw[red](180:1cm) circle (1pt) coordinate (e);
		\filldraw[red](225:1cm) circle (1pt) coordinate (f);
		\filldraw[red](315:1cm) circle (1pt) coordinate (h);	
		%\node at (112.5:1cm) {\tiny$\dots$};
\draw[red] (a) to[bend left] node[right] {\textcolor{black}{\tiny$x_2$}} (b);
\draw[red] (b) to[bend left]   (c);
\draw[red] (d) to[bend left] (e);
\draw[red] (e) to[bend left] node[midway, ,fill=white] {\textcolor{white}{a}} (f);
\draw[red] (a) to[bend right] node[midway, ,fill=white] {\textcolor{white}{a}}  (h);
\draw[red] (e) to node[above] {\textcolor{black}{\tiny$x_{1}$}} (270:0.3cm);
\node at (270:0.3cm) {\textcolor{black}{\tiny$\times$}} ;
\draw[dotted] (108:1cm) to (120:1cm);
\draw[dotted] (194:0.82cm) to (210:0.82cm);
\draw[dotted] (330:0.82cm) to (344:0.82cm);
\draw[black, dashed] (f).. controls (230:0.9cm) and (260:0.9cm) ..(h);
\draw[blue!40!green] (270:0.8cm) ..controls (255:0.25cm) and (0,0) .. (22.5:0.82cm);
\filldraw[blue!40!white](270:0.8cm) circle (1.5pt) coordinate (g);
\node at (288:0.55cm) {\tiny$\gamma_0$};
\node at (350:0.4cm) {\tiny$\gamma_1$};
\end{tikzpicture}
    \caption{Oriented segment $\gamma_0*\gamma_1$}
    \label{fig:special_segment}
\end{figure}
 
Then $\gamma=\gamma_0*\gamma_1* \dots * \gamma_{n}*\gamma_{n+1}$ is a finite arc and by construction $\sigma(\gamma)=\sigma$. Moreover, the grading $\mu$ on $\sigma$ induces a natural grading $f$ on $\gamma$.

To finish the proof, suppose that $\gamma'$ is a finite arcs such that $\sigma(\gamma')=\sigma=\sigma(\gamma)$. Without loss of generality, suppose that $\gamma'$ is in minimal position, then $\sigma(\gamma')$ is already a reduced homotopy string. We claim that $\gamma$ and $\gamma'$ are $O$-homotopic. Any curve is completely determined by the ordered multiset of edges of $\gra^*$ that it crosses and $\sigma(\gamma')=\sigma=\sigma(\gamma)$, thus $\gamma$ and $\gamma'$ have the same multiset of edges of crossings with $\gra^*$ and $\gamma$ and $\gamma'$ are  $O$-homotopic.

The proof for infinite homotopy strings and bands is similar to the above.

\end{proof}

\begin{proof}[Proof of Theorem~\ref{thm:indecomposable objects}]
By \cite[Theorem 3]{BMM03}, the indecomposable objects of the derived category $\mathcal D^b(A)$ are completely described by the graded homotopy strings and bands. The result then follows from Lemma~\ref{lem:from_arcs_to_strings} and Lemma~\ref{lem:from_strings_to_arcs}.
\end{proof}

\section{Applications}

\subsection{Singularity category of a skew-gentle algebra}
The \emph{stable derived category or singularity category} $\SC(A)$ of an algebra $A$ is defined as the Verdier quotient of the bounded derived category with respect to the perfect derived category.

In \cite[Theorem 2.5]{Kal15}, the singularity category of a gentle algebra was explicitly described as a finite product of triangulated orbit categories which turn out to be $n$-cluster categories of type $\mathbb A_1$, as follows. Let $\Lambda=KQ/I$ be a gentle algebra. A cycle $\alpha_1 \dots \alpha_n$ of positive length on $Q$ is \emph{saturated} if each of the length-2 paths $\alpha_{1}\alpha_{2},\ldots,\alpha_{n-1}\alpha_n,\alpha_n\alpha_1$ belongs to $I$.  Let $\mathcal C(\Lambda)$ be the set of cyclical permutation equivalence classes of saturated cycles without repeated arrows. For $n>2$ denote by $\mathcal D^b(K-mod)/\left[n\right]$ the triangulated orbit category as defined in \cite{K05}. Then \cite[Theorem 2.5]{Kal15} shows that $\SC(\Lambda)$ and $\prod_{c\in \mathcal C(\Lambda)} \mathcal D^b(K-mod)/\left[n_c\right]$ are equivalent as triangulated categories, where $n_c$ denotes the length of any cycle in  the equivalence class $c\in \mathcal C(\Lambda)$.

By construction of the dual ribbon graph $\gra_\Lambda$ of a gentle algebra $\Lambda$, it is easy to see that there is a bijection between $\mathcal C(\Lambda)$ and the polygons in $\gra_\Lambda$ with no boundary edges. We will call such polygons \emph{interior polygons} and we denote by $\mathcal{P}^\circ_\Lambda$ the set of interior polygons.  Furthermore, for $P \in \mathcal{P}^\circ_\Lambda$ denote by $n_P$ the number of edges of $P$.  A concrete description of $\SC(\Lambda)$ is then given by 
$$\SC(\Lambda) \simeq \prod_{P \in \mathcal{P}^\circ_\Lambda } \mathcal D^b(K-mod)/\left[n_P\right].$$

The geometric description of the singularity category of a skew-gentle algebra,  follows from the fact that, by \cite[Theorem 3.5]{CL15}, a skew-gentle algebra $A$ and its underlying gentle algebra $\Lambda$ have equivalent singularity categories, and that the generalised ribbon graph of a skew-gentle algebra and that of a gentle algebra have the same number of interior polygons and corresponding interior polygons have the same number of edges. More precisely, we have the following.

\begin{thm}\label{thm:sigularity category}
Let $A$ be a skew-gentle algebra and $(S, M, \orbi, \gra)$ be the orbifold dissection induced by  the generalised ribbon graph $\gra$ of $A$. Then  
$$\SC(A) \simeq \prod_{P \in \mathcal{P}^\circ_A } \mathcal D^b(K-mod)/\left[n_P\right]$$
where $\mathcal{P}^\circ_A$ is the set of interior polygons of $\gra$ and  $n_P$ is the number of edges of $P$, for $P \in \mathcal{\mathcal{P}^\circ_A }$.
\end{thm}

\subsection{Gorenstein dimension of skew-gentle algebras}

Recall that a finite dimensional algebra $A$ is \emph{$d$-Gorenstein} if it has a finite injective dimension $d$ as a left and right $A$-module. Both gentle and skew-gentle algebras are Gorenstein \cite{GR05}.

The following result shows that the Gorenstein dimension of a skew-gentle algebra can be read from its orbifold dissection. For this we recall that a skew-gentle algebra gives rise to an orbifold dissection into generalised polygons which either have no boundary edges or which have exactly one boundary edge. We refer to the latter as a \emph{boundary polygon}. 

\begin{thm}\label{thm:gorenstein dimension}
Let $A$ be a skew-gentle algebra and $(S, M, \orbi, \gra)$ be the orbifold dissection given by the generalised ribbon graph $\gra$ of $A$. Then the Gorenstein dimension of $A$ is equal to $d$,  
where $d-1$ is
the maximal number of internal edges  of  boundary polygons of the dissection, if such boundary polygons exist or zero otherwise.
\end{thm}

\begin{proof}
 Let $\Lambda$ be the gentle algebra obtained from $A$ by deleting all special  special loops. By \cite{GdP99}, there exists a gentle algebra $B$ such that $A$ is Morita equivalent to the skew-group algebra $B*\mathbb{Z}/2$. Then by \cite[Theorem 2.3]{AR91a} and \cite{GR05}, $B$ is Gorenstein and the Gorenstein dimensions of $B$ and $A$ coincide.

By \cite{GR05}, the Gorenstein dimension of $\Lambda$ is equal to the maximal length of saturated paths in $A$ which are not cycles, if such paths exist or zero otherwise. The result follows from  the properties of  $\gra$.
\end{proof}

\subsection{q-Cartan matrices} A classical invariant for graded algebras is the so-called $q$-Cartan matrix which generalises the classical Cartan matrix of a graded finite dimensional algebra, see for example  \cite{Fuller-90}. For this recall that,  $A=KQ/I$ has a grading induced by paths length if $I$ is generated by homogeneous relations. The  $q$-Cartan matrix $C_A(q)=(c_{ij}(q))$ of $A$, for an indeterminate $q$, is the matrix with entries
$$c_{ij}(q)= \sum_n \operatorname{dim}_K (e_iAe_j)_nq^n \in \mathbb Z[q]$$
for vertices $i, j$ in $Q$ and where $(e_iAe_j)_n$ is the component of degree $n$ of $e_iAe_j$. The $q$-Cartan matrix is invariant under graded derived equivalence and specialises to the classical Cartan matrix by setting $q=1$.

In \cite[Theorem 4.2]{BH07}, the determinant of the $q$-Cartan matrix $C_A(q)$ of a (skew-)gentle algebra $A$ is computed in terms of saturated cycles. In the following we show that this description can be read-off the orbifold dissection of $A$.

\begin{thm}\label{thm:q-cartan}
Let $A$ be a skew-gentle algebra and $(S, M, \orbi, \gra)$ be the orbifold dissection induced by the generalised ribbon graph $\gra_A$ of $A$. Denote by $c_k$ the number of (degenerate) interior polygons of $\gra$ with $k$ edges and let $\Lambda$ be the gentle algebra obtained from $A$ by deleting the special loops. Then the $q$-Cartan matrix $C_A(q)$ has determinant
$$ \operatorname{det}C_A(q) = \operatorname{det}C_\Lambda(q)= \prod_{k\geq 1} (1-(-q)^k)^{c_k}.$$
\end{thm}

\section{Example}
In this  section, we will illustrate the geometric model and some of the results in the previous sections on an example. % we consider a non-hereditary skew-gentle algebra $A$ and we compute two complexes in the bounded derived category $\mathcal D^{b}(A)$ of $A$ using grading curves in the orbifold dissection $(S,M, \orbi, \gra)$ of $A$. We also compute the singular category $\mathcal D_{sg}(A)$ of $A$, the Gorenstein dimension of $A$ and the determinant of the $q$-Cartan matrix $C_A(q)$ using the orbifold dissection $(S,M, \orbi, \gra)$ of $A$.

Let $A=KQ/I$ be the skew-gentle algebra with admissible presentation $A^{sg}=KQ^{sg}/I^{sg}$, where $Q$ and $Q^{sg}$ are as follows
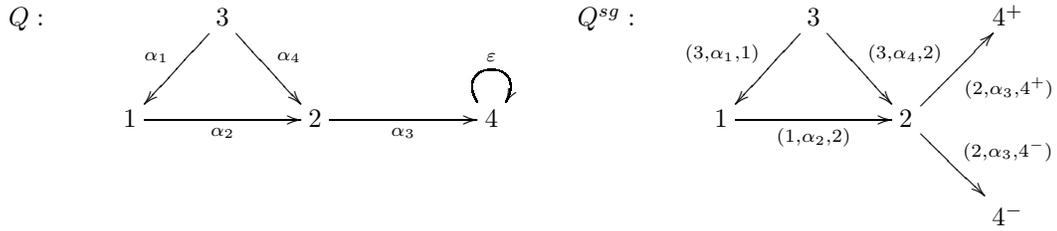
\[
\xymatrix{Q:& \, & 3\ar[dl]_{\alpha_1} \ar[dr]^{\alpha_4}& \, & \,  & \, &  Q^{sg}: &  & 3\ar[dl]_{(3,\alpha_1, 1)} \ar[dr]^{(3,\alpha_4, 2)}& \, & 4^+\\
 & 1 \ar[rr]_{\alpha_2}& \,& 2\ar[rr]_{\alpha_3} & \,  & 4 \ar@(ul,ur)[]^{\varepsilon} && 1 \ar[rr]_{(1,\alpha_2, 2)}& \,& 2\ar[ru]_{(2,\alpha_3, 4^+)} \ar[rd]^{(2,\alpha_3, 4^-)} & \,  & \,\\
& & & & & & & & & & 4^-}
\]
and where $I=\langle \alpha_1\alpha_2, \alpha_4\alpha_3, \varepsilon^2-\varepsilon\rangle$ and $I^{sg}=\langle (3,\alpha_1, 1)(1,\alpha_2,2), (3,\alpha_4, 2)(2,\alpha_3 4^+), (3,\alpha_4, 2)(2,\alpha_3 4^-)\rangle$.

%The admissible presentation $KQ^{sg}/I^{sg}$ of $A$ is given by the quiver
%\[
%\xymatrix@C+1pc{\, & 3\ar[dl]_{(3,\alpha_1, 1)} \ar[dr]^{(3,\alpha_4, 2)}& \, & 4^+\\
%1 \ar[rr]_{(1,\alpha_2, 2)}& \,& 2\ar[ru]_{(2,\alpha_3, 4^+)} \ar[rd]^{(2,\alpha_3, 4^-)} & \,  & \,\\
%& & & 4^-}
%\]
%and $I^{sg}=\langle (3,\alpha_1, 1)(1,\alpha_2,2), (3,\alpha_4, 2)(2,\alpha_3 4^+), (3,\alpha_4, 2)(2,\alpha_3 4^-)\rangle$.
%
%
Following Remark \ref{rem:GeometricRibbonGraph}, the set of vertices $M_A$ of the ribbon graph is the set $\{\alpha_2\alpha_3\epsilon, \alpha_1,\alpha_4, e_4\}$ and its generalised ribbon graph $G_A$ can be seen in Figure~\ref{Fig:final_example_ribbon}. Note that the only difference between a ribbon graph and a generalised ribbon graph is that in a generalised ribbon graph some of the (leaf) vertices, namely those giving rise to orbifold points, are designated to be special. 

\begin{figure}[ht!]
    \centering
\begin{tikzpicture}
\node at (0,0) {$\alpha_2\alpha_3\varepsilon$};
\filldraw (0,-0.2) circle (1pt);
\filldraw (270:1.3) circle (1pt);
\filldraw (270:2.8cm) circle (1pt);
%\filldraw (322:1.3cm) circle (1pt);
\node at (322:1.3cm) {\tiny$\times$};
%\node at (322:1.3cm) [square,draw] {$,$};
\node at (270:1.5cm) {$\alpha_4$};
\node at (270:3.1cm) {$\alpha_1$};
\node at (322:1.6cm) {$e_4$};
\node at (235:1.6cm) {$1$};
\node at (279:0.9cm) {$2$};
\node at (288:2cm) {$3$};
\node at (339:0.9cm) {$4$};
\draw (0,-0.2) to (270:1.3);
\draw (0,-0.2) to[out=180, in=180] (270:2.8cm);
\draw (270:1.3) to[out=0, in=0] (270:2.8cm);
\draw (0,-0.2) to (322:1.3cm);
\end{tikzpicture}
  \caption{Generalised ribbon graphs of algebra $A$}
    \label{Fig:final_example_ribbon}
\end{figure}
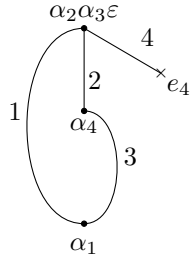

Since $A$ is a skew-gentle algebra, the generalised ribbon graph $G_A$ is embedded in an orbifold with one orbifold point of order 2 which corresponds to the special vertex $e_4$. The corresponding orbifold dissection of $A$ and its dual graph are depicted in Figure\ref{Fig:dissection of G_A}.

%\begin{exmp}
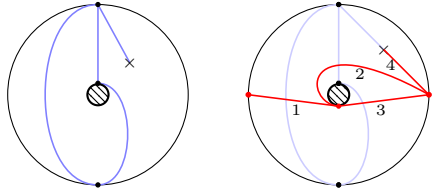
\begin{figure}[ht!]
    \centering
    \begin{tikzpicture}[scale=0.8]
%%%%First figure
%\node at (2cm,-2.5cm) {Orbifold dissection  induced by $G_A$ and its dual graph};
\draw (0,0) circle (1.5cm);
\draw[semithick,blue!50!white] (90:1.5cm) to[out=180,in=180] (270:1.5cm);
\draw[semithick,blue!50!white] (90:0.19cm) to[out=0, in=0] (270:1.5cm);
\draw[semithick, blue!50!white] (90:1.5cm) to (0,0.19);
\draw[semithick, blue!50!white] (90:1.5cm) to (45:0.75cm);
\draw[pattern=north west lines, thick] (0,0) circle(5pt);
%\draw[pattern=north west lines, thick] (270:0.7) circle(5pt);
\filldraw (90:1.5cm) circle (1pt);
\filldraw (270:1.5cm) circle (1pt);
\filldraw (90:0.19cm) circle (1pt);
\node at (45:0.75cm) {\tiny$\times$};
%%%%% SECOND FIGURE
\draw (4,0) circle (1.5cm);
\filldraw (90:1.5cm) ++(4,0)circle (1pt) coordinate (a);
\filldraw (270:1.5cm) ++(4,0)circle (1pt) coordinate (b);
\filldraw (90:0.19cm) ++(4,0)circle (1pt) coordinate (f);
\draw[semithick,blue!20!white] (a) to[out=180,in=180] (b);
\draw[semithick,blue!20!white] (f) to[out=0, in=0] (b);
\draw[semithick, blue!20!white] (a) to (4,0.19cm);
\draw[semithick, blue!20!white] (a) to (4.75, 0.75);
\draw[pattern=north west lines, thick] (4,0) circle(5pt);
\filldraw (a) circle (1pt);
\filldraw (b) circle (1pt);
\filldraw (f) circle (1pt);
\node at (4.75,0.75) {\tiny$\times$};
\filldraw[semithick, red] (180:1.5) ++(4,0) circle (1pt) coordinate (c);
\filldraw[semithick, red] (0:1.5) ++(4,0) circle (1pt) coordinate (d);
%\filldraw[semithick, red] (270:0.9)++(4,0) circle (1pt) coordinate (e);
\draw[semithick, red] (4,-0.19cm) to (c);
\draw[semithick, red] (4,-0.19cm) to (d);
%\draw[semithick, red] (e) to (4,0);
\filldraw[semithick, red] (4,-0.19cm) circle (0.03);
\draw[semithick, red] (4,-0.19cm) ..controls (3.4, -0.15) and (3.35,1.2)..(5.5,0);
\draw[semithick, red] (5.5,0) to (4.75, 0.75);
%\draw[blue!40!green] (270:0.8cm) ..controls (255:0.25cm) and (0,0) .. (22.5:0.82cm);
\node at ($(4,0)+(200:0.75cm)$) {\tiny $1$};
\node at ($(4,0)+(45:0.5cm)$) {\tiny $2$};
\node at ($(4,0)+(30:1cm)$) {\tiny $4$};
\node at ($(4,0)+(340:0.75cm)$) {\tiny $3$};
\end{tikzpicture}
 \caption{Orbifold dissection induced by $G_A$ and its dual graph}
    \label{Fig:dissection of G_A}
\end{figure}

Let $(\gamma, f)$ be the graded curve where $\gamma$ is as in Figure\ref{Fig:curves} and where the grading is given by $f=(1,2,1,0)$. The homotopy string $\sigma(\gamma)$ associated to $\gamma$ is $(\alpha_2\alpha_3)(\overline{\alpha_3})(\overline{\alpha_4})$ and the grading $\mu(f)$ induced by $f$. This induces the following  asymmetric string complex $P^\bullet_{(\sigma(\gamma), \mu(f))}$ in $K^{b,-}(proj$-$A^{sg})$. % is induced by $(\sigma(\gamma), \mu(f)$ is given by the following diagram
\[
\begin{tikzcd}
\, & P_1 \arrow[r, "\alpha_2\alpha_3"] & P_{4^+}\bigoplus P_{4^-} \\
 P_3  \arrow[r, "\alpha_4"]   &  P_2\arrow[ru, "\alpha_3"swap]& \,
\end{tikzcd}
.\]

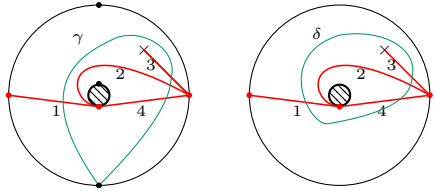
\begin{figure}[ht!]
    \centering
    \begin{tikzpicture}[scale=0.8]
%%% First figure
\draw (0,0) circle (1.5cm);
%\filldraw (90:1.5cm) circle (1pt) coordinate (a);
%\filldraw (270:1.5cm) circle (1pt) coordinate (b);
%\filldraw (90:0.19cm) circle (1pt) coordinate (f);
\draw[pattern=north west lines, thick] (0,0) circle(5pt);
\filldraw (90:1.5cm) circle (1pt);
\filldraw (270:1.5cm) circle (1pt);
\filldraw (90:0.19cm) circle (1pt);
\node at (0.75,0.75) {\tiny$\times$};
\filldraw[semithick, red] (180:1.5) circle (1pt);
\filldraw[semithick, red] (0:1.5) circle (1pt);
%\filldraw[semithick, red] (270:0.9)++(4,0) circle (1pt) coordinate (e);
\draw[semithick, red] (0,-0.19cm) to (180:1.5);
\draw[semithick, red] (0,-0.19cm) to (0:1.5);
\draw[semithick, red] (0.75, 0.75) to (1.5,0);
\filldraw[semithick, red] (0,-0.19cm) circle (1pt);
\draw[semithick, red] (0,-0.19cm) ..controls (-0.6, -0.15) and (-0.65,1.2)..(1.5,0);
\draw[semithick, red] (1.5,0) to (0.75, 0.75);
\draw[blue!40!green] (270:1.5cm) ..controls (170:1.25cm) and (110:0.75) .. (73:1cm);
\draw[blue!40!green] (73:1cm) ..controls (56:1.35cm) and (30:1.5cm) .. (15:1.25cm);
\draw[blue!40!green] (15:1.25cm) ..controls (10:1.20cm) and (350:1.3cm) .. (270:1.5cm);
\node at (200:0.75cm) {\tiny $1$};
\node at (45:0.5cm) {\tiny $2$};
\node at (30:1cm) {\tiny $3$};
\node at (340:0.75cm) {\tiny $4$};
\node at (110:1cm){\tiny$\gamma$};
%%%%%%
%%% Second figure
\draw (4,0) circle (1.5cm);
\filldraw (90:1.5cm) circle (1pt) coordinate (a);
\filldraw (270:1.5cm) circle (1pt) coordinate (b);
\filldraw (90:0.19cm) circle (1pt) coordinate (f);
\filldraw (a) circle (1pt);
\filldraw (b) circle (1pt);
\filldraw (f) circle (1pt);
\node at (4.75,0.75) {\tiny$\times$};
\draw[pattern=north west lines, thick] (4,0) circle(5pt);
\filldraw[semithick, red] (180:1.5) ++(4,0) circle (1pt) coordinate (c);
\filldraw[semithick, red] (0:1.5) ++(4,0) circle (1pt) coordinate (d);
%\filldraw[semithick, red] (270:0.9)++(4,0) circle (1pt) coordinate (e);
\draw[semithick, red] (4,-0.19cm) to (c);
\draw[semithick, red] (4,-0.19cm) to (d);
%\draw[semithick, red] (e) to (4,0);
\filldraw[semithick, red] (4,-0.19cm) circle (1pt);
\draw[semithick, red] (4,-0.19cm) ..controls (3.4, -0.15) and (3.35,1.2)..(5.5,0);
\draw[semithick, red] (5.5,0) to (4.75, 0.75);
\node at ($(4,0)+(200:0.75cm)$) {\tiny $1$};
\node at ($(4,0)+(45:0.5cm)$) {\tiny $2$};
\node at ($(4,0)+(30:1cm)$) {\tiny $3$};
\node at ($(4,0)+(340:0.75cm)$) {\tiny $4$};
%\draw[blue!40!green] (270:0.8cm) ..controls (255:0.25cm) and (0,0) .. (22.5:0.82cm);
\draw[blue!40!green] ($(4,0)+(250:0.5cm)$) ..controls ($(4,0)+(230:0.75cm)$) and ($(4,0)+(150:1.3cm)$) .. ($(4,0)+(90:1cm)$);
\draw[blue!40!green] ($(4,0)+(90:1cm)$) ..controls ($(4,0)+(56:1.35cm)$) and ($(4,0)+(30:1.5cm)$) .. ($(4,0)+(15:1.25cm)$);
\draw[blue!40!green] ($(4,0)+(15:1.25cm)$) ..controls ($(4,0)+(10:1.20cm)$) and ($(4,0)+(350:1.3cm)$) .. ($(4,0)+(250:0.5cm)$);
\node at ($(4,0)+(110:1.1cm)$) {\tiny$\delta$};
\end{tikzpicture}
 \caption{The curves $\gamma$ and $\delta$ in the orbifold dissection $(S,M, \orbi, \gra_A)$}
    \label{Fig:curves}
\end{figure}

Let $(\delta, g)$ be the graded closed curve where $\delta$ is depicted in Figure~\ref{Fig:curves} and $g=(0,1,0,-1,0)$. The asymmetric homotopy band $\sigma(\delta)$ associated to $\delta$ is $(\alpha_2\alpha_3)(\overline{\alpha_3})(\overline{\alpha_4})(\alpha_1)$ and as before $\mu(g)$ is induced by $g$. Let $q(x)\in \operatorname{ind} K[x]$ be a non trivial power of an irreducible polynomial over $K$ with leading coefficient equal to 1 and different from $x$ and $x-1$. The asymmetric band complex $P^\bullet_{(\sigma(\delta), \mu(g), p(x))}$ in $K^{b,-}(proj$-$A^{sg})$ induced by $(\sigma(\delta), \mu(g))$ is 
\[
\begin{tikzcd}
\, & P(1)\otimes K^{\operatorname{deg} q(x)} \arrow[rd, "\alpha_2\alpha_3"]& \\
 P(3) \otimes K^{\operatorname{deg} q(x)} \arrow[rd, "\alpha_4"swap] \arrow[ru, "\alpha_1"]&   & P(4)\otimes K^{\operatorname{deg} q(x)} \\
   &  P(2) \otimes K^{\operatorname{deg} q(x)} \arrow[ru, "\alpha_3"swap]& \,
\end{tikzcd}
\]
where $P(i)=P_i$ for $i\neq 4$ and $P(4)=P_{4^+}\oplus P_{4^-}$.

Observe that the set of interior polygons of $G$ is empty.  Thus by Theorem\ref{thm:sigularity category}, the singularity category $\mathcal D_{sg}(A)$ is equivalent to the category with one element, and by Theorem\ref{thm:q-cartan}, the determinant $\operatorname{det}C_A(q)$ of the $q$-Cartan matrix $C_A(q)$ is zero.

By \ref{thm:gorenstein dimension}, to compute the Gorenstein dimension of $A$, we need to count the maximal number $d$ of internal edges of boundary polygons of the dissection, in this case, the dissection has one digon and two 4-gon, then $d=3$, and as a consequence, the Gorenstein dimension of $A$ is $3-1=2$.

%\begin{figure}[ht]
%\centering
%\begin{tikzcd}
% 1 \arrow[r, "a"] & 2 \arrow[r, "b"] & 3 \\
%& 2 \arrow[ru, "b" swap] \arrow[r,"bc"swap]  & 1 
%\end{tikzcd}
%\end{figure}
%
%Then its associated maps in %$\proj- A^{sg}$ are given by the following diagram.
%\begin{figure}[ht]
%\centering
%\begin{tikzcd}
%&& P(1) \arrow[r, "F_1(a)"] & P(2) \arrow[r, "F_2(b)"] & P_{3^+}\bigoplus P_{3^{-}} && \\
%&&&&&&\\
%&&& P(2) \arrow[ruu, "F_3(b)" swap] \arrow[r,"F_4(bc)"swap]   &  P(1)&&
%\end{tikzcd}
%\end{figure}
%\noindent where $F_1(a)=\begin{bmatrix} a \end{bmatrix}$, $F_2(b)=\begin{bmatrix} b^+ & 0 \end{bmatrix}$, $F_3(b)=\begin{bmatrix}b^+ & b^- \end{bmatrix}$ and $F_4(bc)=\begin{bmatrix}b^+ {}^+c\end{bmatrix}$.
%
%The projective complex $P^{\bullet}_w$ associated to $w$ is given by the following diagram.
%
%
%\begin{figure}[ht]
%\centering
%\begin{tikzcd}
%\dots \arrow[r]& 0\arrow[r] & P(1) \arrow[r, "\partial_1"] & P(2)\bigoplus P(2) \arrow[r, "\partial_2"] &(P_{3^+}\bigoplus P_{3^{-}}) \bigoplus P(1) \arrow[r]& 0\arrow[r] &\dots
%\end{tikzcd}
%\end{figure}
%\noindent where $\delta_1=F_1(a)$ but $\delta_2=\begin{bmatrix}b^+ & 0 &0\\
%b^+ & b^- & b^+ {}^+c
% \end{bmatrix}.$
%
%
% 

%\end{exmp}

\bibliographystyle{acm}
\bibliography{biblio}

\begin{thebibliography}{10}

\bibitem{AB19}
{\sc Amiot, C., and Br{\"u}stle, T.}
\newblock Derived equivalences between skew-gentle algebras using orbifolds.
\newblock {\em arXiv preprint arXiv:1912.04367\/} (2019).

\bibitem{AP17}
{\sc Amiot, C., and Plamondon, P.-G.}
\newblock The cluster category of a surface with punctures via group actions.
\newblock {\em arXiv preprint arXiv:1707.01834\/} (2017).

\bibitem{APS19}
{\sc Amiot, C., Plamondon, P.-G., and Schroll, S.}
\newblock A complete derived invariant for gentle algebras via winding numbers
  and arf invariants.
\newblock {\em arXiv preprint arXiv:1904.02555\/} (2019).

\bibitem{AR91a}
{\sc Auslander, M., and Reiten, I.}
\newblock Cohen-{M}acaulay and {G}orenstein {A}rtin algebras.
\newblock In {\em Representation theory of finite groups and finite-dimensional
  algebras ({B}ielefeld, 1991)}, vol.~95 of {\em Progr. Math.} Birkh\"{a}user,
  Basel, 1991, pp.~221--245.

\bibitem{BMM03}
{\sc Bekkert, V., Marcos, E.~N., and Merklen, H.~A.}
\newblock Indecomposables in derived categories of skewed-gentle algebras.
\newblock {\em Comm. Algebra 31}, 6 (2003), 2615--2654.

\bibitem{BM03}
{\sc Bekkert, V., and Merklen, H.~A.}
\newblock Indecomposables in derived categories of gentle algebras.
\newblock {\em Algebr. Represent. Theory 6}, 3 (2003), 285--302.

\bibitem{BH07}
{\sc Bessenrodt, C., and Holm, T.}
\newblock {$q$}-{C}artan matrices and combinatorial invariants of derived
  categories for skewed-gentle algebras.
\newblock {\em Pacific J. Math. 229}, 1 (2007), 25--47.

\bibitem{Boc16}
{\sc Bocklandt, R.}
\newblock Noncommutative mirror symmetry for punctured surfaces.
\newblock {\em Trans. Amer. Math. Soc. 368}, 1 (2016), 429--469.
\newblock With an appendix by Mohammed Abouzaid.

\bibitem{Bo81}
{\sc Bongartz, K.}
\newblock {T}ilted {A}lgebras.
\newblock In {\em {R}epresentations of {A}lgebras\/} (Puebla, M\'exico, 1981),
  M.~Auslander and E.~Lluis, Eds., no.~903 in Lecture Notes in Mathematics,
  Springer-Verlag, pp.~26--38.

\bibitem{BB80}
{\sc Brenner, S., and Butler, M.}
\newblock {G}eneralization of the {B}ernstein-{G}el'fand-{P}onomarev
  {R}eflection {F}unctors.
\newblock {\em Springer Lecture Notes 832\/} (1980), 103--169.

\bibitem{BD04}
{\sc Burban, I., and Drozd, Y.}
\newblock Derived categories of nodal algebras.
\newblock {\em J. Algebra 272}, 1 (2004), 46--94.

\bibitem{BD17}
{\sc Burban, I., and Drozd, Y.}
\newblock On the derived categories of gentle and skew-gentle algebras:
  homological algebra and matrix problems.
\newblock {\em arXiv preprint arXiv:1706.08358\/} (2017).

\bibitem{Chas-Gadgil}
{\sc Chas, M., and Gadgil, S.}
\newblock The extended {G}oldman bracket determines intersection numbers for
  surfaces and orbifolds.
\newblock {\em Algebr. Geom. Topol. 16}, 5 (2016), 2813--2838.

\bibitem{CL15}
{\sc Chen, X., and Lu, M.}
\newblock Singularity categories of skewed-gentle algebras.
\newblock {\em Colloq. Math. 141}, 2 (2015), 183--198.

\bibitem{Fuller-90}
{\sc Fuller, K.~R.}
\newblock The {C}artan determinant and global dimension of {A}rtinian rings.
\newblock In {\em Azumaya algebras, actions, and modules ({B}loomington, {IN},
  1990)}, vol.~124 of {\em Contemp. Math.} Amer. Math. Soc., Providence, RI,
  1992, pp.~51--72.

\bibitem{GdP99}
{\sc Gei{\ss}, C., and de~la Pe\~{n}a, J.~A.}
\newblock Auslander-{R}eiten components for clans.
\newblock {\em Bol. Soc. Mat. Mexicana (3) 5}, 2 (1999), 307--326.

\bibitem{GLFS16}
{\sc Gei{\ss}, C., Labardini-Fragoso, D., and Schr\"{o}er, J.}
\newblock The representation type of {J}acobian algebras.
\newblock {\em Adv. Math. 290\/} (2016), 364--452.

\bibitem{GR05}
{\sc Gei{\ss}, C., and Reiten, I.}
\newblock Gentle algebras are {G}orenstein.
\newblock In {\em Representations of algebras and related topics}, vol.~45 of
  {\em Fields Inst. Commun.} Amer. Math. Soc., Providence, RI, 2005,
  pp.~129--133.

\bibitem{GH95}
{\sc Green, E., and Huang, R.~Q.}
\newblock Projective resolutions of straightening closed algebras generated by
  minors.
\newblock {\em Adv. Math. 110}, 2 (1995), 314--333.

\bibitem{Gre99}
{\sc Green, E.~L.}
\newblock Noncommutative {G}r\"{o}bner bases, and projective resolutions.
\newblock In {\em Computational methods for representations of groups and
  algebras ({E}ssen, 1997)}, vol.~173 of {\em Progr. Math.} Birkh\"{a}user,
  Basel, 1999, pp.~29--60.

\bibitem{Green2017}
{\sc Green, E.~L.}
\newblock The geometry of strong koszul algebras.
\newblock {\em arXiv preprint arXiv:1702.02918\/} (2017).

\bibitem{GHS17}
{\sc Green, E.~L., Hille, L., and Schroll, S.}
\newblock Algebras and varieties.
\newblock {\em arXiv preprint arXiv:1707.07877\/} (2017).

\bibitem{HKK17}
{\sc Haiden, F., Katzarkov, L., and Kontsevich, M.}
\newblock Flat surfaces and stability structures.
\newblock {\em Publ. Math. Inst. Hautes \'{E}tudes Sci. 126\/} (2017),
  247--318.

\bibitem{Hap87}
{\sc Happel, D.}
\newblock On the derived category of a finite-dimensional algebra.
\newblock {\em Comment. Math. Helv. 62}, 3 (1987), 339--389.

\bibitem{HR82}
{\sc Happel, D., and Ringel, C.}
\newblock {T}ilted {A}lgebras.
\newblock {\em Trans. Amer. Math. Soc. 274\/} (1982), 399--443.

\bibitem{ping-zhou-zhu20}
{\sc He, P., Zhou, Y., and Zhu, B.}
\newblock A geometric model for the module category of a skew-gentle algebra.
\newblock {\em arXiv preprint arXiv:2004.11136\/} (2020).

\bibitem{Kal15}
{\sc Kalck, M.}
\newblock Singularity categories of gentle algebras.
\newblock {\em Bull. Lond. Math. Soc. 47}, 1 (2015), 65--74.

\bibitem{K05}
{\sc Keller, B.}
\newblock On triangulated orbit categories.
\newblock {\em Doc. Math. 10\/} (2005), 551--581.

\bibitem{Kon94}
{\sc Kontsevich, M.}
\newblock Homological algebra of mirror symmetry.
\newblock In {\em Proceedings of the {I}nternational {C}ongress of
  {M}athematicians, {V}ol. 1, 2 ({Z}\"{u}rich, 1994)\/} (1995), Birkh\"{a}user,
  Basel, pp.~120--139.

\bibitem{LF09}
{\sc Labardini-Fragoso, D.}
\newblock Quivers with potentials associated to triangulated surfaces.
\newblock {\em Proceedings of the London Mathematical Society 98}, 3 (2009),
  797--839.

\bibitem{LF16}
{\sc Labardini-Fragoso, D.}
\newblock Quivers with potentials associated to triangulated surfaces, part iv:
  Removing boundary assumptions.
\newblock {\em Selecta Mathematica (New Series) 22}, 1 (2016), 145--189.

\bibitem{LP19}
{\sc Lekili, Y., and Polishchuk, A.}
\newblock Derived equivalences of gentle algebras via fukaya categories.
\newblock {\em Mathematische Annalen\/} (Aug 2019).

\bibitem{Lip15}
{\sc Li, L.}
\newblock Representations of modular skew group algebras.
\newblock {\em Trans. Amer. Math. Soc. 367}, 9 (2015), 6293--6314.

\bibitem{MV07}
{\sc Mart\'{\i}nez-Villa, R.}
\newblock Introduction to {K}oszul algebras.
\newblock {\em Rev. Un. Mat. Argentina 48}, 2 (2007), 67--95 (2008).

\bibitem{OPS18}
{\sc Opper, S., Plamondon, P.-G., and Schroll, S.}
\newblock A geometric model for the derived category of gentle algebras.
\newblock {\em arXiv preprint arXiv:1801.09659\/} (2018).

\bibitem{QZ17}
{\sc Qiu, Y., and Zhou, Y.}
\newblock Cluster categories for marked surfaces: punctured case.
\newblock {\em Compos. Math. 153}, 9 (2017), 1779--1819.

\bibitem{Sch15}
{\sc Schroll, S.}
\newblock Trivial extensions of gentle algebras and {B}rauer graph algebras.
\newblock {\em J. Algebra 444\/} (2015), 183--200.

\bibitem{Sch18}
{\sc Schroll, S.}
\newblock Brauer graph algebras: a survey on {B}rauer graph algebras,
  associated gentle algebras and their connections to cluster theory.
\newblock In {\em Homological methods, representation theory, and cluster
  algebras}, CRM Short Courses. Springer, Cham, 2018, pp.~177--223.

\bibitem{SV}
{\sc Schroll, S., and Valdivieso, Y.}
\newblock Derived category of skew-gentle algebras: morphims between
  indecomposables.
\newblock {\em In preparation\/} (2020).

\bibitem{Smith94}
{\sc Smith, S.~P.}
\newblock Point modules over {S}klyanin algebras.
\newblock {\em Math. Z. 215}, 2 (1994), 169--177.

\bibitem{Val16}
{\sc Valdivieso-D\'{\i}az, Y.}
\newblock Jacobian algebras with periodic module category and exponential
  growth.
\newblock {\em J. Algebra 449\/} (2016), 163--174.

\end{thebibliography}

\end{document}